\documentclass[10pt]{amsart}
\usepackage{amsmath,amsthm,amsfonts,amssymb}
\usepackage{mathrsfs}
\usepackage{stmaryrd}
\usepackage{mathtools}
\usepackage{subfiles}
\usepackage{arydshln}
\usepackage{fullpage}
\usepackage{tikz}
\usepackage{float}
\usepackage{cancel}
\usepackage{xcolor}
\usepackage{setspace}
\usepackage{calligra}
\usepackage{MnSymbol}
\RequirePackage{doi}
\usepackage[all]{xy}
\usepackage{thmtools}
\usepackage[shortlabels]{enumitem}
\usepackage[many]{tcolorbox}
\usepackage{hyperref}
\usepackage[capitalize]{cleveref}

\usetikzlibrary{shapes.geometric}
\usetikzlibrary{calc}

\newtheorem{thm}{Theorem}[section]
\newtheorem{lem}[thm]{Lemma}
\newtheorem{cor}[thm]{Corollary}
\newtheorem{prop}[thm]{Proposition}
\newtheorem{fact}[thm]{Fact}

\newtheorem*{thm*}{Theorem}
\newtheorem*{lem*}{Lemma}
\newtheorem*{cor*}{Corollary}
\newtheorem*{prop*}{Proposition}
\newtheorem*{fact*}{Fact}
\newtheorem*{conj*}{Conjecture}
\newtheorem*{clm*}{Claim}

\theoremstyle{definition}
\newtheorem{defn}[thm]{Definition}
\newtheorem{ex}[thm]{Example}

\newtheorem*{defn*}{Definition}
\newtheorem*{ex*}{Example}
\newtheorem*{xca*}{Exercise}
\newtheorem*{axiom*}{Axiom}

\theoremstyle{remark}
\newtheorem{rmk}[thm]{Remark}

\newtheorem{conv}[thm]{Convention}

\newtheorem*{rmk*}{Remark}
\newtheorem*{setup*}{Setup}
\newtheorem*{cond*}{Condition}
\newtheorem*{cons*}{Construction}
\newtheorem*{conv*}{Convention}
\newtheorem*{obs*}{Observation}
\newtheorem*{ques*}{Question}
\newtheorem*{dis*}{Discussion}

\numberwithin{equation}{section}

\newlist{thmlist}{enumerate}{1}
\setlist[thmlist]{label=(\roman{thmlisti}),
                  ref=\thethm.(\roman{thmlisti}),
                  noitemsep}
                  
\Crefname{listthm}{Theorem}{Theorems}
\Crefname{listlem}{Lemma}{Lemmas}
\Crefname{listprop}{Proposition}{Propositions}
\Crefname{listcor}{Corollary}{Corollaries}
\Crefname{listfact}{Fact}{Facts}
\Crefname{listconj}{Conjecture}{Conjectures}
\Crefname{listclm}{Claim}{Claims}

\addtotheorempostheadhook[thm]{\crefalias{thmlisti}{listthm}}
\addtotheorempostheadhook[lem]{\crefalias{thmlisti}{listlem}}
\addtotheorempostheadhook[prop]{\crefalias{thmlisti}{listprop}}
\addtotheorempostheadhook[cor]{\crefalias{thmlisti}{listcor}}
\addtotheorempostheadhook[fact]{\crefalias{thmlisti}{listfact}}
\addtotheorempostheadhook[clm]{\crefalias{thmlisti}{listclm}}
\addtotheorempostheadhook[conj]{\crefalias{thmlisti}{listconj}}

\DeclareMathOperator{\chr}{char}

\DeclareMathOperator{\Coker}{Coker}

\DeclareMathOperator{\Fr}{Frac}
\DeclareMathOperator{\GL}{GL}

\DeclareMathOperator{\hght}{ht}

\DeclareMathOperator{\Min}{Min}

\DeclareMathOperator{\red}{red}
\DeclareMathOperator{\rk}{rank}

\DeclareMathOperator{\Spec}{Spec}

\newcommand{\mbb}[1]{\mathbb{#1}}

\newcommand{\QQ}{\mbb{Q}}
\newcommand{\NN}{\mbb{N}}

\newcommand{\ZZ}{\mbb{Z}}

\newcommand{\msc}[1]{\mathscr{#1}}

\newcommand{\sJ}{\msc{J}}

\newcommand{\mf}[1]{\mathfrak{#1}}
\newcommand{\mm}{\mf{m}}
\newcommand{\mfa}{\mf{a}}

\newcommand{\mfn}{\mf{n}}

\newcommand{\mfp}{\mf{p}}
\newcommand{\mfq}{\mf{q}}

\newcommand{\mrm}[1]{\mathrm{#1}}
\newcommand{\dd}{\mrm{d}}

\newcommand{\mA}{\mrm{A}}

\newcommand{\mD}{\mrm{D}}

\newcommand{\mF}{\mrm{F}}

\newcommand{\mH}{\mrm{H}}

\newcommand{\mQ}{\mrm{Q}}

\newcommand{\mcal}[1]{\mathcal{#1}}

\newcommand{\cA}{\mcal{A}}

\newcommand{\cF}{\mcal{F}}

\newcommand{\cJ}{\mcal{J}}

\newcommand{\cL}{\mcal{L}}
\newcommand{\cM}{\mcal{M}}

\newcommand{\cP}{\mcal{P}}
\newcommand{\cQ}{\mcal{Q}}
\newcommand{\cR}{\mcal{R}}
\newcommand{\cS}{\mcal{S}}
\newcommand{\cT}{\mcal{T}}

\newcommand{\sheafname}[2]{\mathscr{#1}\text{\kern -3pt {\calligra\large  #2}}\,}

\newcommand{\SDer}{\mathscr{D}\text{\kern -2pt {\calligra\large er}}\,}

\newcommand{\seq}[3][1]{#2_{#1},\dots,#2_{#3}}
\newcommand{\fpd}[2][x]{\frac{\partial #2}{\partial #1}}

\DeclarePairedDelimiter\ang{\langle}{\rangle}

\DeclarePairedDelimiter\ps{\llbracket}{\rrbracket}
\DeclarePairedDelimiter\ls{\llparenthesis}{\rrparenthesis}
\DeclarePairedDelimiter\set{\{}{\}}
\DeclarePairedDelimiter\abs{\lvert}{\rvert}

\makeatletter
\renewcommand*{\eqref}[1]{%
  \hyperref[{#1}]{\textup{\tagform@{\ref*{#1}}}}%
}
\makeatother

\definecolor{zhan}{rgb}{0.0, 0.44, 1.0}

\hypersetup{
  colorlinks = true,
  linkcolor = red,
  anchorcolor = blue,
  citecolor = blue,
  filecolor = cyan,
  menucolor = red,
  runcolor = cyan,
  urlcolor = magenta,
}

\newtcolorbox{mybox}[3][beamer]{
    colback=white,
    coltitle=black,
    colframe=#3,
    coltext=zhan,
    title={\textbf{#2}},
    valign=top,
    halign=left,
    before skip=0.5cm,
    after skip=0.5cm,
    attach title to upper,
    after title={:\ },
    #1
}

\DeclareMathOperator{\loc}{loc}
\newcommand{\tca}[2][\cA]{#2^{*/#1}}
\newcommand{\cl}{\mathsf{cl}}
\newcommand{\fg}{\mathrm{fg}}
\newcommand{\tck}[2][K]{#2^{* #1}}
\newcommand{\dtc}[2][K]{#2^{>*K}}
\newcommand{\ftc}[2][K]{#2^{\mathsf{f}*K}}
\newcommand{\deqtc}[1]{#1^{>*\mathsf{eq}}}
\newcommand{\eqtc}[1]{#1^{*\mathsf{eq}}}
\DeclareMathOperator{\atdeg}{an.tr.deg}
\DeclareMathOperator{\tdeg}{tr.deg}
\DeclareMathOperator{\edim}{edim}
\DeclareMathOperator{\rd}{rd}
\DeclareMathOperator{\bht}{big-height}
\DeclareMathOperator{\Jac}{JRad}
\newcommand{\ddim}{\dim^{*}}

\begin{document}

\title{Test elements in equal characteristic semianalytic algebras}
\author{Zhan Jiang}
\date{\today}
\address{Department of Mathematics\\ University of Michigan\\ Ann Arbor\\ MI 48109}
\email{{\color{zhan}zoeng@umich.edu}}

\begin{abstract}
    We establish a series of results showing that the Jacobian ideal is contained in the test ideal. After proving a new result in characteristic $p$ for complete rings over a field $K$, we extend this containment to equal characteristic 0 for affine, analytic, affine-analytic and semianalytic $K$-algebras. We also prove a result on equiheight algebras essentially of finite type over excellent local rings.
\end{abstract}
\maketitle
\section{Introduction}
Test ideals were first introduced by Mel Hochster and Craig Huneke in their celebrated paper on tight closure \cite{Hochster1990a}. Since their invention, they have found application far beyond their original scope, including Frobenius splittings \cite{Mehta1985,Ramanan1985} and singularity theory \cite{Hochster1994,Hochster1989}. For a good survey on this, we refer to \cite{Schwede2012}. There are various generalizations of test ideals, e.g., to pairs in positive characteristic \cite{Hara2002,Hara2004} and to pairs in mixed characteristic \cite{Ma2018,Schwede2018}. Test ideals are also closely related to multiplier ideals in equal characteristic 0 \cite{Smith2000,Hara2001}.

A key property of test ideals is that they multiply the tight closure of any ideal back into that ideal. The theory has been generalized to arbitrary closure operations in any characteristic \cite{Epstein2019,Perez2019}. There are a great many closure operations \cite{Heitmann2001} and the theory of test ideals of these closure operations remains to be further explored. We aim to contribute to the theory of test ideals for tight closures defined in equal characteristic 0 for (semi-/affine-) analytic algebras \cite{Huneke1999}. In the definition below, the terminology in the first and the third parts (analytic and semianalytic) is taken from \cite[Chapter 13]{Kunz1986}.
\begin{defn}\label{pre-defn-XAnalytic}
Let $R$ be a $K$-algebra where $K$ is a field. In what follows, a \emph{power series algebra} over $K$ refers to a ring of the form $K\ps{\seq{x}{n}}$ for some integer $n\geqslant 0$.
\begin{enumerate}
    \item $R$ is called an \emph{analytic $K$-algebra}, if there is a power series algebra $P$ such that $R$ is finite over $P$.
    \item $R$ is called an \emph{affine-analytic $K$-algebra}, if there is a power series algebra $P$ such that $R$ is of finite type over $P$.
    \item $R$ is called a \emph{semianalytic $K$-algebra}, if there is a power series algebra $P$ such that $R$ is essentially of finite type over $P$.
\end{enumerate}
\end{defn}

We will use the following definition for test ideals in this paper. 
\begin{defn}\cite[Definition 3.1]{Perez2019}\label{pre-defn-TestIdeals}
Let $R$ be a ring and $\cl$ be a closure operation on $R$-modules. The big test ideal of $R$ associated to $\cl$ is defined as
\[
\tau_{\cl}(R) = \bigcap_{N\subseteq M}\left( N:N_M^{\cl} \right)
\]
where the intersection runs over all (not necessarily finitely-generated) $R$-modules $N\subseteq M$. Similarly, we define the finitistic test ideal of $R$ associated to $\cl$ as
\[
\tau_{\cl}^{\fg}(R) = \bigcap_{\substack{N\subseteq M\\ M/N\, \text{ finitely generated}}}\left( N:N_M^{\cl} \right).
\]
\end{defn}

There are two subtleties when working with this definition:
\begin{enumerate}
    \item As we see in \cref{pre-defn-TestIdeals}, there are two kinds of test ideals (the (big) test ideal and the finitistic test ideal).
    \item Test ideals are defined in terms of modules. One can also define them purely in terms of ideals.
\end{enumerate}

For the first point, these two notions associated to tight closure in characteristic $p$ are conjectured to be the same \cite[Conjecture 5.14]{Schwede2012}, and proved to be the same in several cases \cite{Lyubeznik1999,Lyubeznik2001}. Since we only work with tight closure of ideals in noetherian rings, we will stick to the notion of finitistic test ideals. For the second point, test ideals defined in terms of (finitely generated) modules and in terms of all ideals in the ring coincide when the base ring $R$ is approximately Gorenstein \cite[Definition-Proposition 1.1, Definition 1.3]{Hochster1977}. See the remark below.

\begin{rmk}\label{dtc-rmk-TestIdeal}
Let $\cl$ be a closure operation on a noetherian ring $R$ satisfying Semiresiduality and Functoriality \cite[Definition 2.1, 2.2]{Perez2019}. If $R$ is approximately Gorenstein \cite[8.6]{Hochster1990a}, then the finitistic test ideal for modules defined in \cite[Definition 3.1]{Perez2019} coincides with the test ideal for ideals associated with $\cl$, i.e., $\tau^{\fg}_{\cl}(R)=\bigcap_{I\subseteq R} (I:I^{\cl})$ (The argument in \cite[Proposition 8.15]{Hochster1990a} works for any general closure satisfying the Semiresiduality and Functoriality axioms).
\end{rmk}

The condition of being approximately Gorenstein is fairly weak, e.g., when $R$ is either reduced and excellent or when $R$ is of depth 2, $R$ is approximately Gorenstein. In fact, all rings we are working with in this paper are excellent and reduced, hence, approximately Gorenstein. So we make the following convention throughout the paper.

\begin{conv}\label{dtc-conv-TestIdeal}
By the test ideal, we mean the \emph{finitistic test ideal} associated to tight closure in the sense of \cite[Definition 3.1]{Perez2019} in terms of ideals (the case of modules follows from the case of ideals by \cref{dtc-rmk-TestIdeal}). We shall also call the nonzero elements in the test ideal \emph{test elements}.
\end{conv}

Some of the main results of this paper are summarized in the following theorems. See \cref{pre-ssn-TheEuiheightCondition} and, in particular, \cref{pre-defn-GeneralEquiheight},\cref{pre-defn-Equiheight}, and \cref{pre-cor-Equiheight} for the definition and possible presentations of the equiheight condition. See \cref{jcb-defn-AbsoluteReduced} for the definition of absolute reduced-ness.
\begin{thm}[\cref{tep-thm-TEforSemiAnalytic}]
Let $R$ be a semianalytic $K$-algebra that is the localization of an absolutely reduced equiheight affine-analytic $K$-algebra where $\chr(K)=p$. Then the Jacobian ideal $\cJ(R/K)$ is contained in the test ideal of $R$.
\end{thm}
\begin{rmk}
The result above implies the result of \cref{tep-thm-TEforAnalytic}, where $R$ is assumed to be an absolutely reduced equidimensional complete $K$-algebra. However, the proof of \cref{tep-thm-TEforSemiAnalytic} relies on \cref{tep-thm-TEforAnalytic}.
\end{rmk}

\begin{thm}[\cref{tczs-cor-Semianalytic}]
Let $R$ be a semianalytic $K$-algebra that is the localization of an absolutely reduced equiheight affine-analytic $K$-algebra where $\chr(K)=0$. Then the Jacobian ideal $\cJ(R/K)$ is contained in the test ideal of $R$. For any flat $K$-algebra maps $R\to R'$ with geometrically regular fibers, the expansion of the Jacobian ideal $\cJ(R/K)R'$ is contained in the test ideal of $R'$.
\end{thm}
\begin{rmk}
The result above implies the case when $R$ is a reduced equidimensional complete $K$-algebra (\cref{teza-thm-TEforAnalytic}) and the case when $R$ is a reduced equiheight affine-analytic $K$-algebra (\cref{tezs-thm-TEforSemianalytic}). But the proof of \cref{tczs-cor-Semianalytic} relies on \cref{teza-thm-TEforAnalytic} and \cref{tezs-thm-TEforSemianalytic}.
\end{rmk}

\begin{thm}[\cref{teza-cor-EssAffineOverExcellent}]
Let $K$ be a field of characteristic $0$, and let $A$ be a noetherian excellent local $K$-algebra. Let $R$ be an equiheight $A$-algebra essentially of finite type. Suppose that $c \in R$ is an element such that $R_c$ is regular. Then $c$ has a power that is a test element for $K$-tight closure in $R$.
\end{thm}

As indicated above, the key statement that we aim to prove is that the Jacobian ideal is contained in the test ideal for various rings. More explicitly, we first collect some preliminaries in \cref{sn-Pre} and have a discussion on Jacobian ideals in \cref{sn-JacobainIdeal}. Then we prove this statement in characteristic $p$ for absolutely reduced equidimensional complete $K$-algebras (\cref{tep-thm-TEforAnalytic}) and then for semianalytic $K$-algebras that are localizations of absolutely reduced equiheight affine-analytic $K$-algebras (\cref{tep-thm-TEforSemiAnalytic}) in \cref{sn-TEP}. In both \cref{sn-DTC} and \cref{sn-TEZA}, we will focus on rings of equal characteristic 0. In \cref{sn-DTC}, we first present the definition of tight closure in equal characteristic 0 and prove the statement for affine algebras (\cref{HH99}). In \cref{sn-TEZA}, we discuss descent techniques and prove the statement for reduced equidimensional complete local $K$-algebras (\cref{teza-thm-TEforAnalytic}). Then we describe a similar descent process and prove the statement for reduced equiheight affine-analytic $K$-algebras (\cref{tezs-thm-TEforSemianalytic}). Based on \cref{tezs-thm-TEforSemianalytic}, we will establish the statement for semianalytic $K$-algebras that are localizations of reduced equiheight affine-analytic $K$-algebras (\cref{tczs-cor-Semianalytic}). Finally, we prove the statement for noetherian excellent local $K$-algebras (\cref{teza-thm-GeoFlatPower}) and equiheight algebras essentially of finite type over noetherian excellent local $K$-algebras (\cref{teza-cor-EssAffineOverExcellent}).

\subsection*{Acknowledgments}
I would like to heartily thank Mel Hochster for so many helpful discussions and careful readings on the first few drafts. I also want to thank Jack Jeffries and Ilya Smirnov for their comments on the first few drafts.

\section{Preliminaries}\label{sn-Pre}
We discuss some preliminary material about semianalytic algebras, the module of K\"ahler differentials and their relation to regularity. Throughout this section, unless otherwise stated, we do not assume any characteristic constraint. Most material here is covered in \cite{Kunz1986}, but has been reworded in a way that is better adapted to our needs in this paper. We start with some notation.
\begin{defn}\label{pre-defn-Notation}
Let $R$ be a noetherian ring and $\mfp\in\Spec(R)$ a prime ideal of $R$. Let $M$ be an $R$-module and $I\subseteq R$ an ideal of $R$.
\begin{enumerate}
\item $\Min(R)$ is the set of minimal primes of $R$.
\item $\bht(I)$ is the big height of $I$, i.e., $\bht(I)=\max\set{\hght(Q)\mid {Q\in\mrm{Min}(I)}}$. 
    \item If $I\subseteq \mfp$, then the notion $\hght_{\mfp}(I)$ represents the smallest height of a prime $P$ such that $I\subseteq P\subseteq \mfp$. Note that this is equal to $\hght(IR_{\mfp})$ in $R_{\mfp}$.
    \item $\dim_{\mfp}R$ is the supremum of the lengths of all chains of prime ideals containing $\mfp$. We have $\dim_{\mfp}R=\dim R/\mfp + \dim R_{\mfp}$.
    \item The \emph{regularity defect} of $R$ at $\mfp$ is defined to be $\rd_{\mfp}(R):=\edim(R_{\mfp})-\dim (R_{\mfp})$, where $\edim(R_{\mfp})$ is the embedded dimension of the noetherian local ring $R_{\mfp}$.
    \item The residue field of $\mfp$ is denoted by $\kappa_{\mfp}(R) := R_{\mfp}/\mfp R_{\mfp}$.
    \item $\mu(M)$ is the minimal number of generators of $M$ and $\mu_{\mfp}(M)$ is the minimal number of generators of the localized module $M_{\mfp}$.
    \item If $A$ is a matrix with entries in $R$, then $\rk_{\mfp}(A)$ represents the determinantal rank of the matrix $A$ over the residue field $\kappa_{\mfp}(R)$ at prime $\mfp$.
\end{enumerate}
\end{defn}

\subsection{Derivations and universal extensions}
A good reference here is Chapters 1-3, 11-12, in \cite{Kunz1986}.
\begin{defn}
Let $R_0$ be a ring, $R$ an $R_0$-algebra and $M$ an $R$-module.
\begin{enumerate}
    \item A \emph{$R_0$-derivation} of $R$ in $M$ is an $R_0$-linear map $\delta: R\to M$ that satisfies the Leibniz rule, i.e., for all $a,b\in R$, $\delta(ab)=a\delta b+b\delta a$. In the case $R_0=\ZZ$, $\ZZ$-derivations of $R$ are simply called \emph{(absolute) derivations} of $R$ \cite[1.1]{Kunz1986}.
  \item An $R_0$-derivation $\dd:R\to M$ is called \emph{universal} if for any $R_0$-derivation $\delta:R\to N$ there is one and only one $R$-linear map $\ell:M\to N$ such that $\delta=\ell\circ \dd$ \cite[1.18]{Kunz1986}.
  \item If $\dd: R\to M$ is a universal $R_0$-derivation of $R$ then the $R$-module $M$ is denoted by $\Omega_{R/R_0}$ and is called \emph{the module of (K\"ahler) differentials} of $R$ over $R_0$. The universal derivation of $R$ over $R_0$ is sometimes denoted by $\dd_{R/R_0}$ \cite[1.20]{Kunz1986}. It is well-known that the module of K\"ahler differentials exists. If we denote the kernel of the canonical map $R\otimes_{R_0}R\to R$ by $I$, then $\Omega_{R/R_0}\cong I/I^2$ \cite[1.7, 1.21]{Kunz1986}.
  \item For an $R_0$-derivation $\delta:R\to M$, we shall write $R\delta R$ for the submodule of $M$ generated by $\set{\delta r}_{r\in R}$. If $\dd$ is the universal $R_0$-derivation of $R$, then $R\dd R=\Omega_{R/R_0}$ \cite[1.21 (b)]{Kunz1986}.
  \end{enumerate}
\end{defn}

Chapters 2 and 3 in \cite{Kunz1986} describe differential algebras and universal extensions of differential algebras. Here, we work with derivations, which are more relevant to our needs. These two languages can be translated to and from one another in most cases.
\begin{defn}\cite[1.24]{Kunz1986}
Let $R$ and $S$ be two $R_0$-algebras and $\rho:R\to S$ an $R_0$-algebra morphism. Let $\delta:R\to M$ be an $R_0$-derivation of $R$ and $\delta':S\to M'$ an $R_0$-derivation of $S$. Then $\delta'$ is called an \emph{extension} of $\delta$ if there exists an $R$-linear map $\varphi:M\to M'$ such that
\[
\xymatrix{
R\ar[r]^{\rho}\ar[d]_{\delta} & S\ar[d]^{\delta'} \\
M\ar[r]^{\varphi} & M'
}
\]
  is commutative. An extension $\delta'$ of $\delta$ is called \emph{universal} if any other extensions $\Delta:S\to N$ of $\delta$ can be uniquely written as a specialization of $\delta'$, i.e., there exists a unique $S$-linear map $\phi:M'\to N$ such that $\Delta=\phi\circ\delta'$.
\end{defn}

The most important result about universal extensions of an $R_0$-derivation $\delta:R\to R\delta R$ is that they exist \cite[3.20]{Kunz1986}. From the definition we see that the universal extension $\delta':S\to M'$ of $\delta$ is unique up to canonical isomorphism. We write $\Omega_{S/\delta}$ for $M'$ and call $\Omega_{S/\delta}$ the \emph{module of K\"ahler differentials of $S$ over $\delta$}. If $\delta$ is the trivial derivation, i.e., $R\delta R=0$, then $\Omega_{S/\delta}=\Omega_{S/R}$ is the usual module of K\"ahler differentials of $S$ over $R$.

The module of K\"ahler differentials is finitely generated for affine $R_0$-algebras. Finite generation is important when we define the Jacobian ideal in \cref{sn-JacobainIdeal}. However, modules of K\"ahler differentials are not necessarily finitely generated for power series rings over $R_0$ (or more generally, semianalytic $K$-algebras, see \cref{pre-defn-XAnalytic}). So we have the following definitions.
\begin{defn}
\begin{enumerate}
    \item An $R_0$-derivation $\delta:R\to R\delta R$ of $R$ is called \emph{finite} if $R\delta R$ is finitely generated as an $R$-module.
    \item $\dd:R\to M$ is called \emph{universal finite} if $\dd$ is finite and each finite $R_0$-derivation $\delta$ of $R$ factors through $\dd:R\to M$ with respect to an unique $R$-homomorphism \cite[11.1]{Kunz1986}. If such $\dd:R\to M$ exists, then $M$ is unique up to a canonical $R$-isomorphism. We write $\widetilde{\Omega}_{R/R_0}$ for $M$ and call it the \emph{universal finite module of differentials of $R$ over $R_0$} (frequently referred to in the literature as the \emph{universally finite} module of differentials).
    \item Let $\rho:R\to S$ be a homomorphism of $R_0$-algebras and $\delta:R\to R\delta R$ a derivation of $R/R_0$. An $R_0$-derivation $\dd:S\to N$ of $S$ into an $S$-module $N$ is called a \emph{universal finite $\rho$-extension of $\delta$}, if the following hold:
\begin{enumerate}
\item $\dd$ is a $\rho$-extension of $\delta$ and finite (i.e., $S\dd S$ finitely generated)
\item If $\Delta:S\to N'$ is an arbitrary finite $\rho$-extension of $\delta$, then there is exactly one $S$-linear map $h:N\to N'$ with $\Delta:=h\circ \dd$.
\end{enumerate}
If the universal finite $\rho$-extension $\dd:S\to N$ of $\delta$ exists, we write $N:=\widetilde{\Omega}_{S/\delta}$ and call this the \emph{universal finite module of differentials of $S/\delta$}. In case $\delta$ is the trivial derivation of $R$ we write $\widetilde{\Omega}_{S/R}$ instead of $\widetilde{\Omega}_{S/\delta}$ \cite[11.4]{Kunz1986}.
\end{enumerate}
\end{defn}

The most important result here is that under mild assumptions, universal finite modules of differentials exist. See, for example, \cite[12.5]{Kunz1986}.

\subsection{The equiheight condition}\label{pre-ssn-TheEuiheightCondition}
Throughout this section, let $\cR$ be a ring of finite type over a noetherian local ring $(\cA,\mm)$. We want to introduce an equiheight condition. Let us start with characterizing different types of primes in $\cR$ with respect to $\cA$.
\begin{defn}\label{pre-defn-PrimeTypes}
  A prime ideal $Q$ in $\cR$ is called
\begin{itemize}
    \item \emph{special} with respect to $\cA$ if $Q+\mm \cR$ is a proper ideal of $\cR$;
    \item \emph{typical} with respect to $\cA$ if $Q+\mm \cR=\cR$ is the whole ring.
\end{itemize}
More generally, we will say that $\cR$ is \emph{special} with respect to $\cA$ if $\mm\cR\neq\cR$ and \emph{typical} with respect to $\cA$ if $\mm\cR=\cR$.
\end{defn}

Note that the set of special maximal ideals is closed since it is the set of maximal primes containing $\mm\cR$: if $Q$ is maximal, since $Q+\mm\cR$ is a proper ideal containing $Q$, it must be $Q$. Hence, $\mm\cR\subseteq Q$. The set of typical maximal ideals is open. Based on whether $\cR$ is special or typical, we introduce a modified dimension notion as below.

\begin{defn}
  If $\cR$ is a domain. We define $\ddim_{\cA} (\cR)$ to be $\dim(\cR)$ if $\cR$ is special and $\dim(\cR)+1$ if $\cR$ is typical. We also define
  \[
  \ddim_{\cA}(\cR)=\sup\set{\ddim_{\cA}(\cR/\mfp):\mfp\in\Min(\cR)}.
  \]
\end{defn}

To see the usefulness of this notion, we need to investigate the structure of $\cA$. More precisely, we have the following lemma.
\begin{lem}\label{pre-lem-HilbertRing}
If $\cA$ is equidimensional of dimension $n$, and $a \in \mm$ is not in any minimal prime,  then $\cA_a$ has dimension $n-1$, and $\cA_a$ is a Hilbert ring, i.e., every prime (hence, every radical ideal) can be written as an intersection of maximal ideals.  

If in addition $\cA$ is catenary, then all maximal ideals of $\cA_a$ have height $n-1$.
\end{lem}
\begin{proof}

We first show that $\dim \cA_a$ is $n-1$. Since $\cA_a$ is a localization of $\cA$, the dimension cannot go up, i.e., $\dim\cA_a\leqslant \dim\cA$. If there is a prime chain of length $n$ in $\cA_a$, the preimage of it is a prime chain of length $n$ contained in $\mm$ but not containing $a$, which implies that $\mm$ has height $n+1$. Hence, we have $\dim \cA_a<n$. 

To see that $\dim\cA_a = n -1$, we construct a sequence $a_1, \dots, a_{n-1}$ in $\mm$ iteratively using prime avoidance. Since $a$ is not in any minimal prime of $\cA$, we have $\hght(a) = 1$. We choose $a_1 \in \mm$ such that $a_1$ avoids all minimal primes of $\cA$ and all minimal primes of $(a)$. Because $n \geqslant 2$ (the $n=1$ case trivially  yields dimension $0$), $\mm$ is not contained in the union of these primes. This choice guarantees $\hght(a_1) = 1$ and $\hght(a, a_1) = 2$.

Inductively, assume $a_1, \dots, a_{k-1}$ have been chosen such that $\hght(a_1, \dots, a_{k-1}) = k-1$ and $\hght(a, a_1, \dots, a_{k-1}) = k$. Provided $k < n$, $\mm$ is not contained in any minimal prime of $(a_1, \dots, a_{k-1})$ or of $(a, a_1, \dots, a_{k-1})$. By prime avoidance, we may choose $a_k \in \mm$ avoiding all such minimal primes. This guarantees $\hght(a_1, \dots, a_k) = k$ and $\hght(a, a_1, \dots, a_k) = k+1$.

Continuing this process, we obtain $a_1, \dots, a_{n-1}$ such that $\hght(a_1, \dots, a_{n-1}) = n-1$ and $\hght(a, a_1, \dots, a_{n-1}) = n$. Let $Q$ be any minimal prime of $(a_1, \dots, a_{n-1})$. By Krull's height theorem, $\hght(Q) \leqslant n-1$. Since $\hght(a_1, \dots, a_{n-1}) = n-1$, we must have $\hght(Q) = n-1$. If $a \in Q$, then $Q$ would contain the ideal $(a, a_1, \dots, a_{n-1})$, which has height $n$, contradicting $\hght(Q) = n-1$. Thus, $a \nin Q$. It follows that $Q\cA_a$ is a prime ideal in $\cA_a$ of height $n-1$, establishing that $\dim \cA_a = n-1$.

To show that $\cA_a$ is a Hilbert ring, we prove that any prime $P$ in $\cA_a$ is an intersection of maximal ideals in $\cA_a$. Equivalently, we show that the intersection of maximal ideals in $\cA_a/P$ is zero. Let $\cP$ be the preimage of $P$ in $\cA$. Then $\cA_a/P=(\cA/\cP)_a$. We replace $\cA$ by $\cA/\cP$, and then we aim to show that the intersection of all maximal ideals in $\cA_a$ is zero when $\cA$ is a domain. Let $0\neq b \in \cA$ be a non-unit in $\cA_a$. Extend $ab$ to a system of parameters $ab,\seq{b}{n-1}$ for $\cA$. Choose a minimal prime $Q'$ of $(\seq{b}{n-1})$. Since $Q'$ does not contain $ab$, $\cA/Q'$ has dimension 1. By the argument above, $(\cA/Q')_a$ has dimension exactly one less than $\cA/Q'$, so $(\cA/Q')_a=\cA_a/Q'\cA_a$ is a field. Thus, $Q'\cA_a$ is maximal, and does not contain $b$. Hence, the intersection of all maximal ideals in $\cA_a$ is zero.

Given any maximal $\mm'$ ideal of $\cA_a$, it contains a minimal prime of $\cA$ expanded to $\cA_a$. We can kill that minimal prime and assume that we are in the domain case. Let $\cM$ be the preimage of $\mm'$ in $\cA$. Then $a\neq 0$ in $\cA/\cM$. Hence $\dim \cA/\cM\geqslant 1$. $\cM$ is one of the maximal ideals that do not contain $a$, so $\dim \cA/\cM = 1$. Because $\cA$ is assumed to be both catenary and equidimensional of dimension $n$, the dimension formula yields $\hght(\cM) + \dim(\cA/\cM) = \dim(\cA) = n$. Therefore, $\hght(\cM) = n - 1$ and $\hght(\mm') = n - 1$.
\end{proof}

\begin{rmk}
 By an example in \cite[Appendix A]{nagata1975local},  there is a local domain $\cA$ of dimension 3 that has a prime $\cQ$ such that $\hght\cQ = 1$ and $\dim(\cA/\cQ) = 1$. This provides a local domain of dimension 3 with saturated chains of length 2 and of length 3. Thus $\cA$ has a prime $\cQ'$ such that $\hght \cQ' = 2$ and $\dim(\cA/\cQ') = 1$.  So if $a \in \mm - (\cQ \bigcup \cQ')$,  then $\cA_a$ has maximal ideals of height 1 and of height 2. Therefore the ``catenary'' condition in \cref{pre-lem-HilbertRing} cannot be omitted.
\end{rmk}

\begin{lem}\label{pre-defn-PrimeStructureInR}
Assume that $\cA$ is universally catenary and $\cR$ is a domain.
\begin{enumerate}
    \item If $\cR$ is typical, then all maximal ideals of $\cR$ have height $\dim(\cR)$, which is also $\ddim_{\cA}(\cR)-1$.
    \item If $\cR$ is special, then there is at least one maximal ideal of $\cR$ containing $\mm\cR$, i.e., that is special, and all special maximal ideals have height $\dim(\cR)$, while all typical maximal ideals have height $\dim(\cR)-1$.
\end{enumerate}
In both cases,
\begin{itemize}
    \item typical maximal ideals have height $\ddim_{\cA}(\cR)-1$ and special maximal ideals have height $\ddim_{\cA}(\cR)$;
    \item $\ddim_{\cA}(\cR)=\dim(\overline{\cA})+\tdeg_{\Fr(\overline{\cA})}(\Fr(\cR))$.
\end{itemize}
where $\overline{\cA}$ is the image of $\cA$ in $\cR$. 
\end{lem}
\begin{proof}
Let $\mm$ be the maximal ideal of $\cA$ and $K$ be its residue field, i.e., $K=\cA/\mm$. We will replace $(\cA,\mm)$ by their images $(\overline{\cA},\overline{\mm})$ in $\cR$ and still write them as $(\cA,\mm)$. For any maximal ideal $Q$ of $\cR$, write $P$ for its contraction in $\cA$. By dimension formula, we have
\begin{equation}\label{pre-eq-DimensonFormula}
\hght(Q) -\hght(P) = \tdeg_{\Fr(\cA)}(\Fr(\cR)) - \tdeg_{\kappa_P(\cA)}(\kappa_Q(\cR)).
\end{equation}

We aim to prove the following
  \begin{itemize}
  \item if $Q$ is typical, then $\hght(Q)=\dim (\cA)+\tdeg_{\Fr(\cA)}(\Fr(\cR)) - 1$;
  \item If $Q$ is special, then $\hght(Q) = \dim (\cA)+\tdeg_{\Fr(\cA)}(\Fr(\cR))$.
  \end{itemize}
  
  If $Q$ is typical, then we claim that $P$ is of dimension 1, i.e., of height $\dim(\cA)-1$. Since $Q+\mm \cR=\cR$, $P\neq \mm$. By the generalized Noether normalization theorem, $\cR/Q$ is module-finite over $(\cA/P)_f$ for some $f\in \cA/P$. Since $\cR/Q$ is a field, we deduce that $(\cA/P)_f$ is a field. But $\cA/P$ is not a field. Hence, $P$ has height $\dim \cA-1$ and $\tdeg_{\kappa_P(\cA)}(\kappa_Q(\cR))=0$. By \eqref{pre-eq-DimensonFormula}, $\hght(Q)=\dim(\cA)-1+\tdeg_{\Fr(\cA)}(\Fr(\cR))$. If $\cR$ is typical, then all maximal ideals are typical. Hence, all maximal ideals of $\cR$ are of the same height $\dim(\cR)=\ddim_{\cA}(\cR)-1$.

If $Q$ is special, then $P=\mm$. Hence, $\hght(P)=\dim \cA$ and $\tdeg_{\kappa_P(\cA)}(\kappa_Q(\cR))=0$. By \eqref{pre-eq-DimensonFormula}, $\hght(Q)=\dim(\cA)+\tdeg_{\Fr(\cA)}(\Fr(\cR))$. If $\cR$ is special, then $\cR$ has at least one special maximal ideal, which is of height $\dim(\cA)+\tdeg_{\Fr(\cA)}(\Fr(\cR))=\dim(\cR)$. The typical maximal ideals of $\cR$, if any, is of height $\dim(\cR)-1$ by the argument above.
\end{proof}

\begin{cor}\label{pre-cor-PrimeHeightLocalized}
Assume $\cA$ is universally catenary. Let $\overline{\cA}$ be the image of $\cA$ in $\cR$. Then
\begin{enumerate}
    \item If $\cR$ is special, $\dim(\cR_a)=\dim(\cR)-1$ for all $a\in \mm \overline{\cA}-\set{0}$. On the other hand, if $\cR$ is typical then $\dim(\cR_a)=\dim(\cR)$ for all $a\in \mm \overline{\cA}-\set{0}$.
    \item Let $f\in \cR-\set{0}$. Then $\ddim_{\cA}(\cR)=\ddim_{\cA}(\cR_f)$. If $\cR$ is typical, then $\dim(\cR)=\dim(\cR_f)$ as well. If $\cR$ is special, then $\dim(\cR_f)=\dim(\cR)$ if $\mm \cR_f\neq \cR_f$, and $\dim(\cR_f)=\dim(\cR)-1$ if $\mm \cR_f=\cR_f$.
\end{enumerate}
\end{cor}
\begin{proof}
(1) By \cref{pre-lem-HilbertRing}, $\overline{\cA}_a$ is a Hilbert ring of dimension $\dim(\overline{\cA})-1$. The result follows by applying the formula in \cref{pre-defn-PrimeStructureInR} to $\cR$ over $\overline{\cA}$ and $\cR_a$ over $\overline{\cA}_a$.

(2) Since $\Fr(\cR)=\Fr(\cR_f)$ and $\cR_f$ is still an affine domain over $\overline{\cA}$, \cref{pre-defn-PrimeStructureInR} implies that $\ddim_{\cA}(\cR)=\ddim_{\cA}(\cR_f)$. The rest follows from this equality.
\end{proof}

\begin{defn}\label{pre-defn-GeneralEquiheight}
We define $\cR$ to be \emph{equiheight} over $\cA$ if $\ddim_{\cA}(\cR/\mfp)$ is independent of the choice of $\mfp\in\Min(\cR)$.
\end{defn}

Right from the definition, we see that
\begin{cor}
Assume that $\cA$ is universally catenary. If $\cR$ is equiheight over $\cA$, then $\cR$ is equiheight if and only if $\cR_{\red}$ is equiheight.
\end{cor}

To justify the name of this condition a little bit, we note that if all minimal primes of $\cR$ are typical, then this condition asserts that $\dim(\cR/\mfp)+1=\dim(\cR)+1$. While if all minimal primes are special, then this condition is equivalent to $\dim(\cR/\mfp)=\dim(\cR)$.

The corollary below follows directly from \cref{pre-cor-PrimeHeightLocalized}.
\begin{cor}
Assume that $\cA$ is universally catenary. If $\cR$ is equiheight over $\cA$, then $\cR_f$ is equiheight for any $f\in R$ that is not nilpotent.
\end{cor}

\begin{prop}
Assume that $\cA$ is universally catenary and $\cR$ is equiheight over $\cA$.
Then for every prime ideal $P$ of $\cR$,
\begin{equation}\label{pre-eq-HeightFormula}
    \ddim_{\cA}(\cR/P)+\hght(P)=\ddim_{\cA}(\cR).
\end{equation}
\end{prop}
\begin{proof}
We first assume that $\cR$ is a domain over $\cA$. Suppose that $P$ is typical. Let $\cM$ be a maximal ideal containing $P$, then $\cM$ must be typical. Since $\cR$ is universally catenary, $\hght(\cM)=\hght(\cM/P)+\hght(P)$. By \cref{pre-defn-PrimeStructureInR}, $\hght(\cM)=\ddim_{\cA}(\cR)-1$ and $\hght(\cM/P)=\ddim_{\cA}(\cR/P)-1$. Hence, \eqref{pre-eq-HeightFormula} holds. If $P$ is special, we choose $\cM$ to be a maximal ideal containing $P+\mm\cR$. By \cref{pre-defn-PrimeStructureInR}, $\hght(\cM)=\ddim_{\cA}(\cR)$ and $\hght(\cM/P)=\ddim_{\cA}(\cR/P)$. \eqref{pre-eq-HeightFormula} also holds.

For the general case, let $\mfp$ be a minimal prime of $\cR$ such that $\hght(P/\mfp)=\hght(P)$. Applying \eqref{pre-eq-HeightFormula} to $\cR/\mfp$, we have 
\begin{align*}
\ddim_{\cA}((\cR/\mfp)/(P/\mfp))+\hght(P/\mfp)&=\ddim_{\cA}(\cR/\mfp) \\
\Rightarrow \ddim_{\cA}(\cR/P)+\hght(P)&=\ddim_{\cA}(\cR/\mfp).
\end{align*}
By definition of $\cR$ being equiheight,  $\ddim_{\cA}(\cR/\mfp)=\ddim_{\cA}(\cR)$ and \eqref{pre-eq-HeightFormula} follows.
\end{proof}

\begin{cor}\label{pre-cor-EquiHeightJustified}
Assume that $\cA$ is universally catenary and $\cR$ is equiheight over $\cA$. Then $\cR/I$ is equiheight if and only if all minimal primes of $I$ have the same height. In particular, this holds whenever $\cR/I$ is a domain.
\end{cor}
\begin{proof}
Let $P_1,P_2$ be two minimal primes of $I$. By definition of equiheight of $\cR/I$, we have $\ddim_{\cA}(\cR/P_1)=\ddim_{\cA}(\cR/P_2)$, which translates to $\hght(P_1)=\hght(P_2)$ by \eqref{pre-eq-HeightFormula}.
\end{proof}

\begin{rmk}
 Suppose that $\cR$ is affine over a noetherian complete local ring. Then $\cR$ may be written as $\cA[\seq{z}{n}]/I$, where $\cA$ is regular local domain. Then $\cR$ is equiheight if and only if all minimal primes of $I$ have the same height, independent of the choice of presentation, by \cref{pre-cor-EquiHeightJustified}.
\end{rmk}

We next study the equiheight property under various kinds of ring extensions.
\begin{lem}\label{pre-lem-AlgExtDomain}
Assume that $\cA$ is universally catenary and $\cR$ is domain. If a domain $\cS$ over $\cA$ is an algebraic extension of $\cR$, then $\ddim_{\cA}(\cS)=\ddim_{\cA}(\cR)$. If $\cS$ is module-finite over $\cR$, $\dim(\cS)=\dim(\cR)$ as well.
\end{lem}
\begin{proof}
The second part is a standard fact. The first part follows from the second bullet point of \cref{pre-defn-PrimeStructureInR}, since $\overline{\cA}$ and the transcendence degree do not change.
\end{proof}

\begin{prop}\label{pre-prop-ModFinTorFree}
Assume that $\cA$ is universally catenary. If $\cR$ is reduced and equiheight over $\cA$ , then so is every module-finite extension $\cS$ that is torsion-free over $\cR$ in the sense that nonzerodivisors in $\cR$ are nonzerodivisors in $\cS$ (this condition holds if $\cS$ is $\cR$-flat).
\end{prop}
\begin{proof}
For any minimal prime $\mfq$ of $\cS$, the retraction $\mfp$ in $\cR$ is also a minimal prime. If not, $\mfp$ must contain a nonzerodivisor in $\cR$, which is also a nonzerodivisor in $\cS$. Then $\mfq$ cannot be minimal. Then $\cS/\mfq$ is a module-finite extension of $\cR/\mfp$, and by \cref{pre-lem-AlgExtDomain}, $\ddim_{\cA}(\cS/\mfq)=\ddim_{\cA}(\cR/\mfp)$ is independent of $\mfq$.
\end{proof}

\begin{prop}\label{pre-prop-LocalDescriptionEquiHeight}
Assume that $\cA$ is universally catenary. Suppose that $\Spec(\cR)$ is connected and has an open cover by sets $D(f)$ such that $\cR_f$ is equiheight over $\cA$. Then $\cR$ is equiheight over $\cA$.
\end{prop}
\begin{proof}
Let $\set{D(f_i)}_{1\leqslant i\leqslant m}$ be a finite subcover. If $D(f_i)$ and $D(f_j)$ overlap, then there is a minimal prime $\mfp$ lying in their intersection. Then
\[
\ddim_{\cA}(\cR_{f_i})=\ddim_{\cA}(\cR_{f_i}/\mfp \cR_{f_i})=\ddim_{\cA}(\cR_{f_if_j}/\mfp \cR_{f_if_j})=\ddim_{\cA}(\cR_{f_j}/\mfp \cR_{f_j})=\ddim_{\cA}(\cR_{f_j}).
\]
Thus, $\ddim_{\cA}(\cR/\mfq\cR)$ is constant as $\mfq$ runs through minimal primes of $D(f_i)\cup D(f_j)$. Consider equivalence relation generated by $D(f_i)\sim D(f_j)$ if $D(f_i)\cap D(f_j)\neq\emptyset$. Then unions of equivalence classes are mutually disjoint open sets. Since $\Spec(\cR)$ is connected, there is only one equivalence class, and it follows that $\ddim_{\cA}(\cR/\mfq)$ is the same for all $\mfq\in\Min(\cR)$.
\end{proof}

\begin{prop}
Assume that $\cA$ is universally catenary and $\cR$ is equiheight over $\cA$. If $\cS$ is smooth over $\cR$ with $\Spec(\cS)$ connected, then $\cS$ is equiheight over $\cA$.
\end{prop}
\begin{proof}
If $\Spec(\cR)$ is not connected, break it into connected components $\cR\cong \cR_1\times \cdots\times \cR_h$ where each $\Spec(\cR_i)$ is connected. Since $\Spec(\cS)$ is connected, its image must lie in one of the $\Spec(\cR_i)$. Then we have a morphism $\cR_i\to \cS$ (algebraically, we obtain such map by localizing at the idempotent $e_i$ of $\cR$ where $\cR_{e_i}\simeq \cR_i$). So we may assume that $\Spec(\cR)$ is connected.

By a structure theorem for smooth morphisms \cite[Chapter III, Proposition 3.1]{Iversen1973}, we can find finitely many elements $g_j \in \cR$, $f_j$ in $\cS$ such the $f_jg_j$ generate the unit ideal of $\cS$, the $g_j$ generate the unit ideal of $\cR$, and each $\cS_{f_jg_j}$ is a special smooth extension of $\cR_{g_j}$, i.e., \'etale over a polynomial extension. By \cref{pre-prop-LocalDescriptionEquiHeight}, it suffices to prove the polynomial and \'etale extension case. When we adjoin indeterminates there is a bijection of minimal primes, and $\ddim_{\cA}(\cR)$ increases by the number of indeterminates, which is constant because of connectedness. Hence, it suffices to prove the result for the \'etale extension case. 

Let us assume that $\cS$ is \'etale over $\cR$. Again by structure theorem (\cite[Chapter III, Theorem 2.1]{Iversen1973}), we can find finitely many elements $g_j \in \cR$, $f_j$ in $\cS$ such the $f_jg_j$ generate the unit ideal of $\cS$, the $g_j$ generate the unit ideal of $\cR$, and each $\cS_{f_jg_j}$ is a standard \'etale extension of $\cR_{g_j}$. By \cref{pre-prop-LocalDescriptionEquiHeight}, it suffices to prove the standard \'etale case.

Let us assume that $\cS$ is standard \'etale over $\cR$. This implies that $\cS$ is a localization at one element of $\cR[z]/H$, where $z$ is an indeterminate and $H(z)$ is monic. But then $\cR[z]/H(z)$ is module-finite and free over $\cR$, so that it is equiheight by \cref{pre-prop-ModFinTorFree}, and localization does not affect this.
\end{proof}

If we assume that in addition that $\cA$ is excellent, we will prove that $\cR$ is an equiheight $\cA$-algebra if and only if $\widehat{\cA}\otimes_{\cA}\cR$ is an affine equiheight algebra over $\widehat{\cA}$ (see part (4) below).

\begin{thm}\label{pre-thm-EquiheightPreserved}
Assume that $\cA$ is excellent. Then:
\begin{enumerate}
    \item $\dim(\widehat{\cA}\otimes_{\cA}\cR)=\dim(\cR)$.
    \item For every $\mfq \in \Min(\widehat{\cA}\otimes_{\cA}\cR)$ with contraction $\mfp$ to $\cR$, $\ddim_{\cA}(\widehat{\cA}\otimes_{\cA}\cR/\mfq)=\ddim_{\cA}(\cR/\mfp)$.
    \item $\ddim_{\cA}(\widehat{\cA}\otimes_{\cA}\cR)=\ddim_{\cA}(\cR)$.
    \item $\cR$ is equiheight over $\cA$ if and only if $\widehat{\cA}\otimes_{\cA}\cR$ is equiheight over $\widehat{\cA}$.
\end{enumerate}
\end{thm}
\begin{proof}
Let $\cS:=\widehat{\cA}\otimes_{\cA}\cR$. Since $\widehat{\cA}$ is faithfully flat over $\cA$, $\cS$ is faithfully flat over $\cR$ and the minimal primes of $\cR$ coincide with the contractions of the minimal primes of $\cS$ to $\cR$. Hence, both (3) and (4) follows from (2). Therefore we need only to prove (1) and (2).

Proof of (1): since $\cR\to\cS$ is faithfully flat, $\dim(\cR)\leqslant\dim(\cS)$. Therefore, for (1) we need only prove $\dim(\cR)\geqslant\dim(\cS)$. For this we will use noetherian induction on $\cA$. Let $\mfq$ be a minimal prime of $\cS$ with contraction $\cP$ in $\cA$ such that $\dim(\cS)=\dim(\cS/\mfq)$. It suffices to prove that $\dim(\cS/\mfq)=\dim(\cS/\cP\cS)\leqslant\dim(\cR/\cP\cR)$. Hence, we can replace $\cA,\cR,\cS$ by $\cA/\cP,\cR/\cP\cR,\cS/\cP\cS$ respectively. So we assume that $\cA$ is a domain with $\cA\subseteq\cR$. By Noether normalization for domains, we may choose a nonzero element $a\in \cA$ such that $\cR_a$ is module-finite extension of a polynomial ring $\cA_a[\seq{x}{d}]$ over $\cA_a$. By the lemma of generic freeness, we may replace $a$ by a multiple of $a$ and assume that $\cR_a$ is $\cA_a$-free over. Since $\cS_a\cong \widehat{\cA}\otimes_{\cA}\cR_a$, $\cS_a$ is an $\widehat{\cA}_a$-free module-finite extension of $\widehat{\cA}_a[\seq{x}{d}]$. Hence $\dim(\cS_a)=\dim(\widehat{\cA}_a[\seq{x}{d}])=\dim(\widehat{\cA}_a)+d$. Since $\cA$ is a an excellent domain, $\widehat{\cA}$ is reduced and equidimensional of the same dimension as $\cA$. It follows that $\dim(\cA_a)=\dim(\cA)-1$ since $a$ is a nonzerodivisor. Thus, $\dim(\cS_a)=\dim(\cA_a)+d=\dim(\cA_a[\seq{x}{d}])=\dim(\cR_a)$. By the hypothesis of noetherian induction, $\dim(\cR/a\cR)=\dim(\cS/a\cS)$. Therefore, $\dim(\cR)=\dim(\cS)$.

Proof of (2): let $\mfq$ be a minimal prime of $\cS$ with contraction $\mfp$ to $\cR$. Let $\cP$ be the contraction to $\cA$. We can replace $\cA,\cR,\cS$ by $\cA/\cP,\cR/\mfp,\cS/\mfp\cS$ respectively. Hence, we assume that $\mfp = 0$ and we aim to prove that $\ddim_{\cA}(\cS/\mfq)=\ddim_{\cA}(\cR)$. By Noether normalization, we know that $\cR_a$ is module-finite and torsion-free over $\cA_a[\seq{x}{d}]$. Thus $\cR_a$ embeds in a finitely generated free module $G$ over $\cA_a[\seq{x}{d}]$. It follows that $\cS_a$ embeds in the $\widehat{\cA}_a[\seq{x}{d}]$-free module $\widehat{\cA}\otimes G$, and so $\cS_a$ is torsion-free over $\widehat{\cA}_a[\seq{x}{d}]$. Note that $a\nin\mfq$, so $\mfq\cS_a$ is a minimal prime of $\cS_a$. It lies over a minimal prime $\mfq_0$ of $\widehat{\cA}_a[\seq{x}{d}]$, since the contraction cannot consist of any nonzerodivisor. This corresponds via expansion to a minimal prime $\pi$ of $\widehat{\cA}$ that does not contain $a$. Thus $\cS_a/\mfq\cS_a$ is module-finite over $(\widehat{\cA}/\pi)[\seq{x}{d}]$. Thus
\begin{align*}
    \ddim_{\widehat{\cA}}(\cS/\mfq) & = \ddim_{\widehat{\cA}}(\cS/\mfq)_a \qquad \text{by \cref{pre-cor-PrimeHeightLocalized}(2)} \\
    & = \ddim_{\widehat{\cA}}(\widehat{\cA}/\pi)[\seq{x}{d}] \qquad \text{by \cref{pre-lem-AlgExtDomain}} \\
    &=\dim(\widehat{\cA}/\pi)+d \qquad \text{by the second bullet point in \cref{pre-defn-PrimeStructureInR}}\\
    &=\dim(\cA)+d  \qquad \text{since $\widehat{\cA}$ is equidimensional and $\dim(\widehat{\cA})=\dim(\cA)$} 
\end{align*}
Similarly, this is $\ddim_{\cA}(\cA_a[\seq{x}{d}])=\ddim_{\cA}(\cR)$, as required.
\end{proof}

\subsection{Structure of semianalytic algebras}
The key result here is the following proposition.
\begin{prop}\cite[13.4]{Kunz1986}\label{pre-prop-UniqueMaximalAnalytic}
  Any reduced semianalytic $K$-algebra $R$ contains a unique maximal analytic $K$-algebra $A$, i.e., all $K$-subalgebras of $R$ that are analytic $K$-algebras are contained in $A$. If $A'$ is an arbitrary analytic $K$-algebra with $A'\subseteq R$ such that $R$ is essentially of finite type over $A'$, then $A$ is the integral closure of $A'$ in $R$.
\end{prop}

For a reduced semianalytic $K$-algebra $R$ we denote by $\mA(R)$ the maximal analytic subalgebra of $R$. If $R$ is not reduced, such an algebra need not exist. If $R$ is a domain, then $\mA(R)$ is a local domain, because $\mA(R)$ is always a direct product of local rings. We also have that the $K$-algebra maps between reduced semianalytic $K$-algebras are compatible with taking the maximal analytic $K$-subalgebra. That is, if $\varphi:R\to S$ is a homomorphism of reduced semianalytic $K$-algebras, then $\varphi(\mA(R))\subseteq \mA(S)$. Hence $\varphi$ induces a $K$-homomorphism $\mA(\varphi):\mA(R)\to \mA(S)$ \cite[13.5]{Kunz1986}.

For each reduced semianalytic $K$-algebra $R$, let $\widetilde{\Omega}_{R/K}$ denote the universal $R$-extension of $\widetilde{\Omega}_{\mA(R)/K}$, the universal finite differential module of $\mA(R)$ over $K$, whose existence is guaranteed by \cite[12.9]{Kunz1986}. If $A'$ is an arbitrary analytic $K$-algebra such that $R$ is essentially of finite type over $A'$, then by the transitive law for universal extensions, $\widetilde{\Omega}_{R/K}$ is also the universal $R$-extension of $\widetilde{\Omega}_{A'/K}$.

\begin{rmk}
 In \cite{Kunz1986}, Kunz uses the notion $\mD_K(R)$ to refer to the universal $R$-extension of $\widetilde{\Omega}_{\mA(R)/K}$. He uses $\Omega_{R/\delta}$ for any $\delta:A\to A\delta A$ if $R$ is not reduced and $R$ is essentially of finite type over $A$ (note that in this case $\mA(R)$ is not well-defined). Since we are only working with reduced semianalytic algebras, there is no ambiguity in using $\widetilde{\Omega}_{R/K}$.
\end{rmk}

We also need the definition of ``analytic transcendence degree'' of field extensions.
\begin{defn}
Let $K$ be a field and $\seq{X}{n}$ indeterminates over $K$. The field $F:=K\ls{\seq{X}{n}}$ will always denote the fraction field of the power series ring $K\ps{\seq{X}{n}}$. Let $L$ be a field extension of $K$. 
\begin{itemize}
    \item $L$ is called a \emph{semianalytic extension field} of $K$ if there is a $K$-homomorphism $F\to L$ such that $L$ is finitely generated over $F$. If $L$ is a finite extension of $F$, we call $L$ an \emph{analytic extension field} of $K$.
    \item Let $L$ be an analytic field extension of $K$. Suppose $L$ is finite over $F\subseteq L$. Then $n$ is called the \emph{analytic transcendence degree} of $L$ over $K$, $n:=\atdeg(L/K)$, and $\set{\seq{X}{n}}$ is called an \emph{analytic transcendence basis} of $L$ over $K$. The basis $\set{\seq{X}{n}}$ is called \emph{separating}, if $L$ is separable over $F$. $L$ is called \emph{analytically separable} over $K$, if $L$ has a separating analytic transcendence basis over $K$ \cite[13.7]{Kunz1986}. Note that the number $n$ above is an invariant of the field extension $L$ over $K$, as it is the Krull dimension of the maximal analytic algebra $\mA(L)$.
\end{itemize}
\end{defn}

\subsection{Primes in affine-analytic algebras}
Next, we want to make use of the discussions in \cref{pre-ssn-TheEuiheightCondition}: Let $R$ be a reduced affine-analytic $K$-algebra. By \cref{pre-prop-UniqueMaximalAnalytic} we have a canonical choice of the base ring $\cA$, that is, $\mA(R)$. Hence, we make the following definition.
\begin{defn}\label{pre-defn-TypeOfPrimes}
  Let $R$ be a reduced affine-analytic $K$-algebra. Let $A:=\mA(R)$ and let $\Jac(A)$ denote the Jacobson radical of $A$. A prime ideal $Q$ in $R$ is called \emph{special} if $Q+\Jac(A) R$ is a proper ideal of $R$, and is \emph{typical} if $Q+\Jac(A) R=R$ is the whole ring. We shall write $\ddim(R)$ for $\ddim_{\mA(R)}(R)$.
\end{defn}

\begin{rmk}\label{pre-rmk-DetectTypeOfPrimes}
  Under the assumption of \cref{pre-defn-TypeOfPrimes}, since $\mA(R)$ is module-finite over some complete local ring, $\mA(R)$ is a product of several complete local rings, i.e., $\mA(R)=A_1\times \cdots\times A_s$. Each $A_i$ is a complete local domain with a maximal ideal denoted by $\mm_i$. \cref{pre-defn-TypeOfPrimes} is the same as \cref{pre-defn-PrimeTypes} if we restrict to each $A_i$.  If we have a presentation 
  \[
  R=T/I \text{ where }T=K\ps{\seq{x}{n}}[\seq{z}{m}],
  \]
then each $\mm_i$ is radical of the image of $(\seq{x}{n})$ in $A_i$. Thus, the Jacobson radical $\Jac(\mA(R))=\mm_1\times \cdots\times \mm_s$ is the radical of the image of $(\seq{x}{n})$ in $\mA(R)$. Since for any prime ideal $Q$, $Q+\Jac(\mA(R)) R$ is a proper ideal if and only if $Q+(\seq{x}{n})R$ is a proper ideal, we may use $(\seq{x}{n})R$ to detect the type of a prime ideal of $R$.
\end{rmk}

By applying \cref{pre-defn-PrimeStructureInR} to $T=K\ps{\seq{x}{n}}[\seq{z}{m}]$, we have the following corollary.
\begin{cor}\label{pre-cor-HeightOfPrimes}
Let $T=K\ps{\seq{x}{n}}[\seq{z}{m}]$. Then special maximal ideals in $T$ have height $\ddim(T)=n+m$, and typical maximal ideals in $T$ have height $\ddim(T)-1=n+m-1$. In particular, $T$ is equiheight.
\end{cor}

We will call a reduced affine-analytic $K$-algebra $R$ equiheight if is satisfies the \cref{pre-defn-GeneralEquiheight}. More precisely,
\begin{defn}\label{pre-defn-Equiheight}
Let $R$ be a reduced affine-analytic $K$-algebra. Then $R$ is equiheight if $\ddim(R/\mfp)$ is independent of the choice of $\mfp\in\Min(R)$.
\end{defn}

By \cref{pre-cor-EquiHeightJustified}, we have the following corollary.
\begin{cor}\label{pre-cor-Equiheight}
  Let $R$ be a reduced affine-analytic $K$-algebra. Then $R$ is equiheight if and only if for any presentation $R=T/I$ where $T=K\ps{\seq{x}{n}}[\seq{z}{m}]$ and $I\subseteq T$ an ideal, $I$ is of pure height, i.e., all minimal primes of $I$ have the same height in $T$.
\end{cor}

\subsection{Height of ideals in regular rings}
First we state the following Serre's intersection theorem, see \cite[Chapitre V, B.6, Th\'eor\`eme 3]{Serre1975}.
\begin{thm}[Serre's intersection theorem]\label{pre-thm-SerreIntersec}
  Let $A$ be a regular ring and $P,Q$ be two prime ideals in $A$ such that $P+Q$ is a proper ideal. Then we have
  \[
    \hght(P)+\hght(Q)\geqslant\hght(P+Q).
  \]
\end{thm}

We immediately see that
\begin{cor}\label{pre-cor-SerreIntersec}
  Let $I,J$ be two ideals in a regular local ring $A$. Then we have
  \[
  \hght(I)+\hght(J)\geqslant \hght(I+J).
  \]
\end{cor}
\begin{proof}
Choose $P$ minimal over $I$ such that $\hght(P)=\hght(I)$ and $Q$ minimal over $J$ such that $\hght(Q)=\hght(J)$. Since we are in the local case, the sum $P+Q$ is contained in the maximal ideal of $A$. By \cref{pre-thm-SerreIntersec}, $\hght(P)+\hght(Q)\geqslant \hght(P+Q)$. Since $P+Q\supseteq I+J$, we have $\hght(P+Q)\geqslant \hght(I+J)$. So $\hght(I)+\hght(J)\geqslant \hght(I+J)$.
\end{proof}

The following intersection theorem of Serre is well-known to experts. Since we cannot find a solid source, we include a proof here.
\begin{thm}\label{pre-thm-RegHeight}
  Let $R$ be a noetherian regular ring and let $I\subseteq R$ be an ideal. Let $h=\bht(I)$. Let $R\to S$ be a ring homomorphism between noetherian rings. If $IS$ is a proper ideal, then $\hght(IS)\leqslant h$.
\end{thm}

Before giving the proof of \cref{pre-thm-RegHeight}, we will need the following ``Cohen factorization theorem'' \cite[Theorem 1.1]{Avramov1994}.

\begin{thm}[Cohen factorization]\label{pre-thm-CohenFactorization}
  Let $\varphi:R\to S$ be a local ring map between noetherian local rings, and assume that $S$ is complete. Then there exists a complete local ring $T$ and two local ring maps $\tau:R\to T$ and $\theta:T\to S$ such that
  \begin{enumerate}
  \item $\varphi = \theta\circ \tau$ and $\theta:T\twoheadrightarrow S$ is a surjection,
  \item $\tau$ is flat and $T/\mm_RT$ is regular where $\mm_R$ is the maximal ideal of $R$.
  \end{enumerate}
  Such a decomposition is called a \emph{Cohen factorization}.
\end{thm}
\begin{rmk}
  Since the map $\tau$ in the \cref{pre-thm-CohenFactorization} is flat and local, it is, in fact, faithfully flat.
\end{rmk}

\begin{proof}[Proof of \cref{pre-thm-RegHeight}]
  Suppose for contradiction that $\hght(IS)>h$. Since $IS\subseteq S$ is proper, there is a minimal prime $Q$ of $IS$ in $S$ such that $\hght(Q)=\hght(IS)$. Then $\hght(IS_Q)=\hght(IS)=\dim(S_Q)$. Since $I$ will generate a $Q$-primary ideal in the completion, we see that $\hght(IS_Q)=\hght(I\widehat{S_Q})$. 
  
  To apply Cohen factorization, we must pass to a complete local ring. Let $\mathfrak{p}$ be a minimal prime of $\widehat{S_Q}$ such that $\dim(\widehat{S_Q}/\mathfrak{p}) = \dim(\widehat{S_Q})$. Let $S' = \widehat{S_Q}/\mathfrak{p}$. Then $S'$ is a complete local domain, and the extended ideal $IS'$ still satisfies $\hght(IS') > h$. Let $\mm$ be the contraction of the maximal ideal of $S'$ to $R$. Because $IS'$ is proper, $I \subseteq \mm$. Since $I \subseteq \mm$, we can choose a minimal prime $P$ of $I$ such that $I \subseteq P \subseteq \mm$ and $\hght(P) \leqslant h$. 
  
  Localizing $R$ at $\mm$ does not change the height of $P$, and the map $R \to S'$ naturally induces a local ring homomorphism $R_\mm \to S'$. Because $S'$ is a complete local ring, we may apply \cref{pre-thm-CohenFactorization} to this map.
  
  By the Cohen factorization theorem (\cref{pre-thm-CohenFactorization}), there exists a complete regular local ring $T$ and a map $R_\mm \to T \twoheadrightarrow S'$ such that $R_\mm \to T$ is faithfully flat. Because the map is faithfully flat, $\hght(PT)=\hght(P_{\mm})\leqslant h$. Since $S'$ is a domain, the kernel of the map $T\twoheadrightarrow S'$ is a prime ideal $P'$. Note that the image of $PT+P'$ is primary for the maximal ideal in $S'$. So $\hght(PT+P')-\hght(P')=\dim(S')>h$. On the other hand, by \cref{pre-thm-SerreIntersec}, we have $\hght(PT)+\hght(P')\geqslant \hght(PT+P')$. This implies $\hght(PT)>h$, which contradicts the fact that $\hght(PT)\leqslant h$.
\end{proof}

\section{Regularity and Jacobian Ideals}\label{sn-JacobainIdeal}
\subsection{Absolute regularity}
In \cite{Kunz1986}, Kunz discusses absolute regularity for analytic algebras. Here we extend the notion to affine-analytic algebras. First we have the following definition.
\begin{defn}
Every analytic $K$-algebra $A$ is a finite extension of a power series algebra $K\ps{\seq{X}{d}}\subseteq A$. We call $K\ps{\seq{X}{d}}\hookrightarrow A$ a \emph{Noether normalization} of $A$.  
\end{defn}

The following is a modification of \cite[14.10]{Kunz1986}.
\begin{defn}
Let $R$ be a reduced affine-analytic algebra over a field $K$ and let $L$ be a field extension of $K$. A \emph{constant field extension} $R_L$ of $R$ with $L$ is an affine-analytic $L$-algebra $R_L$ for which there is a local $K$-homomorphism $R\to R_L$ satisfying the following universal property: if $\beta:R\to S$ is any local $K$-homomorphism into an affine-analytic $L$-algebra $S$, then there is exactly one $L$-homomorphism $\gamma:R_L\to S$ such that $\beta=\gamma\circ \alpha$.
\end{defn}
The existence is easily shown: Let $A:=\mA(R)$ be the unique maximal analytic $K$-subalgebra of $R$. Then we have a Noether normalization $K\ps{\seq{X}{d}}\hookrightarrow A$. For any field extension $L$ of $K$, then tensor product $R_L:=L\ps{\seq{X}{d}}\otimes_{K\ps{\seq{X}{d}}} R$ is the constant field extension of $R$ with $L$. Alternatively, since the constant field extension of analytic algebras is constructed in the paragraph below \cite[14.10]{Kunz1986}, we can form $A_L$ and then $R_L:=A_L\otimes_A R$.

We extend results in \cite[14.11]{Kunz1986} to the case of affine-analytic rings. For this purpose, we need the following lemma.

\begin{lem}\label{jcb-lem-HeightUnderFlat}
  Let $T$ be flat over $S$ and $I\subseteq S$ a proper ideal such that all minimal primes have height $h$. If $IT\neq T$ (this is automatic when $S\to T$ is faithfully flat), then all minimal primes of $IT$ have the same height $h$.
\end{lem}
\begin{proof}
  Let $Q$ be a minimal prime of $IT$ in $T$, and let $P$ be its contraction in $S$. So $S_P\to T_Q$ is a faithfully flat map. Since $Q$ is minimal over $IT$, $T_Q/IT_Q$ has dimension zero. By base change, this ring is faithfully flat over $S_P/IS_P$. So $S_P/IS_P$ has dimension zero. Hence $P$ is minimal over $I$ and therefore it has height $h$. Since $\dim(T_Q)=\dim(S_P)+\dim(T_Q/PT_Q)$ and $T_Q/IT_Q\twoheadrightarrow T_Q/PT_Q$ have dimension zero, we conclude that $\dim(T_Q)=\dim(S_P)=h$.
\end{proof}

\begin{prop}\label{jcb-prop-ConstFieldExt}
Let $R$ be a reduced affine-analytic $K$-algebra and let $A:=\mA(R)$. Let $K\ps{\seq{X}{d}}\hookrightarrow A$ be a Noether normalization of $A$, and $L/K$ a field extension.
\begin{enumerate}
\item For any ideal $I$ of $R$, $(R/I)_L=R_L/IR_L$.
\item $R_L$ is faithfully flat over $R$.
\item $\dim R_L=\dim R$.
\item If $R$ is equiheight, so is $R_L$.
\item For any $\mfq\in\Spec(R)$ there is a $\mfp\in\Spec(R_L)$ such that $\dim_{\mfp} R_L\geqslant \dim_{\mfq} R$.
\item $\widetilde{\Omega}_{R_L/L}\cong R_L\otimes_R \widetilde{\Omega}_{R/K}$.
\end{enumerate}
\end{prop}
\begin{proof}
  (1): Let $J=I\cap A$. Then by \cite[14.11(b)]{Kunz1986}, we have $(A/J)_L=A_L/JA_L$. Since $A/J=\mA(R/I)$, we have 
  \[
  (R/I)_L=(A/J)_L\otimes_{A/J} (R/I)=(A_L/JA_L)\otimes_{A/J}(R/I)=A_L\otimes_A (R/I)= R_L/IR_L.
  \]
  
  (2): Since $L\ps{\seq{X}{d}}$ is faithfully flat over $K\ps{\seq{X}{d}}$, the base-changed map $R\to R_L$ is faithfully flat as well.
  
  (3): Let $\set{\mfq_i}_{1\leqslant i\leqslant s}$ be the set of minimal primes of $R$. Then $\cap_{j=1}^s \mfq_j = 0$. Hence $\cap_{j=1}^s \mfq_j R_L = 0$. So the dimension of $R_L$ is the supremum of the dimensions of $R_L/\mfq_j R_L \cong (R/\mfq_j)_L$. So we can base change to $R/\mfq_i$ for some $i$ without changing $\dim R$ and $\dim R_L$. Now we assume that $R$ is a domain, $A_0=K\ps{\seq{x}{n}}\subseteq R$ and $B_0=L\ps{\seq{x}{n}}$. Then $R_L=R\otimes_{A_0} B_0$. By (2), we know that $R\hookrightarrow R_L$ and $\dim R_L\geqslant \dim R$. So we only need to show that $\dim R_L\leqslant \dim R$. Note that nonzerodivisors on $R$ are also nonzerodivisors on $R_L$. Hence $R_L$ is $R$-torsion free. If $\mfp$ is a minimal prime of $R_L$, then $\mfp\cap R={0}$. Choose $\mfp$ minimal such that $\dim R_L=\dim R_L/\mfp$. Write $R'=R_L/\mfp$ and $C_0=B_0/(\mfp\cap B_0)$. We have $R\hookrightarrow R'$ and $A_0\hookrightarrow C_0$. Since $R$ generates $R_L$ over $B_0$, it likewise generates $R'$ over $C_0$. So we can choose a (finite) set of elements in $R$ which will be a transcendence basis for $\Fr(R')$ over $\Fr(C_0)$. Then the same set of elements must be algebraically independent over $\Fr(A_0)$. Suppose that this finite set has $t$ elements. Then
  \begin{align*}
      \dim(R) & \geqslant \dim(A_0)+t-\varepsilon_{A_0,R} \quad \text{by the last formula in \cref{pre-defn-PrimeStructureInR}} \\
      &=\dim(B_0) +t-\varepsilon_{A_0,R} \\
      & \geqslant \dim(C_0)+t-\varepsilon_{A_0,R} \\
      &\geqslant \dim(C_0)+t-\varepsilon_{C_0,R'}=\dim(R')
  \end{align*}
  where $\varepsilon_{A_0,R}$ is $1$ if $(\seq{x}{n})R=R$ and $0$ otherwise. Note that if $\varepsilon_{A_0,R}=1$ then $\varepsilon_{C_0,R'}=1$. So this is proved.
  
  (4): Write $R=T/I$ where $T=K\ps{\seq{x}{n}}[\seq{z}{m}]$. Then by \cref{pre-cor-Equiheight}, we assume that $I$ has pure height $h$. Therefore by \cref{jcb-lem-HeightUnderFlat}, $IT'$ has pure height $h$ as well, where $T'=L\ps{\seq{x}{n}}[\seq{z}{m}]$. Since $R_L=T'/IT'$, we conclude that $R_L$ is equiheight.
  
  (5): Let $\mm$ be a maximal ideal in $R$ such that $\dim_{\mm}R=\dim_{\mfq}R$. There is a prime ideal $\mm'$ in $R_L$ lying over $\mm$. Since $R_{\mm}\to (R_L)_{\mm'}$ is still faithfully flat, we have
  \[
  \dim_{\mm'} R_L \geqslant \dim (R_L)_{\mm'} \geqslant \dim R_{\mm}=\dim_{\mm} R =\dim_{\mfq} R.
  \]
  
  (6): Since the universal finite module of K\"ahler differentials is calculated as the cokernel of the Jacobian matrix, the conclusion follows directly.
\end{proof}

We generalize the definition of absolute regularity \cite[14.12]{Kunz1986} to the affine-analytic case.
\begin{defn}\label{jcb-defn-AbsReg}
An affine-analytic $K$-algebra $R$ is called \emph{absolutely regular} at $\mfq\in\Spec(R)$, if for any field extension $L/K$ and any $\mfp\in\Spec(R_L)$ with $\mfp\cap R=\mfq$ the local ring $(R_L)_{\mfp}$ is regular.
\end{defn}
The notion of absolute regularity is equivalent to the notion of regularity when $K$ is of characteristic 0.

We also define absolute reducedness as follows.
\begin{defn}\label{jcb-defn-AbsoluteReduced}
  An affine-analytic $K$-algebra $R$ is called \emph{absolutely reduced}, if for any field extension $L/K$, $R_L$ is reduced.
\end{defn}

An affine-analytic $K$-algebra $R$ is absolutely reduced if and only if $R_P$ is absolutely regular for any minimal prime $P$. 

\subsection{Jacobian ideals}
Let $R$ be a ring.
\begin{defn}[Fitting ideals]
  For any finitely presented $R$-module $M$, let $R^m\overset{(a_{ij})}{\longrightarrow} R^n\to M$ be a presentation of $M$. The $i$th \emph{fitting ideal} is the ideal generated by the minors of size $n-i$ of the matrix $(a_{ij})$.
\end{defn}
The Fitting ideals do not depend on the choice of generators and relations of $M$. Here, we also use the convention that the $i$th fitting ideal is the whole ring $R$ if $n-i\leqslant 0$, and the zero ideal if $n-i >\min\set{n,m}$. For more about Fitting ideals, \cite[Appendix D]{Kunz1986} is a good reference.

Let $S$ be an $R$-algebra and $\delta:R\to R\delta R$ be a derivation such that $\Omega_{S/\delta}$ is a finitely presented $S$-module. We have the following definition.
\begin{defn}\cite[10.1]{Kunz1986}
The $i$th Fitting ideal $\sJ_i(S/\delta):=\mF_i(\Omega_{S/\delta})$ of $\Omega_{S/\delta}$ is called the $i$th Jacobian ideal of $S/\delta$. In case $\delta$ is the trivial derivation of $R$, we write $\sJ_i(S/R)$ for $\sJ_i(S/\delta)$.
\end{defn}
Clearly, we have 
\begin{prop}
Under the assumptions of the definition above,
\begin{enumerate}
    \item $\sJ_0(S/\delta)\subseteq \sJ_1(S/\delta)\subseteq \sJ_2(S/\delta)\subseteq \cdots$.
    \item $\sJ_i(S/\delta)=S$ for $i\geqslant\mu(\Omega_{S/\delta})$.
\end{enumerate}
\end{prop}

\begin{defn}
  If $S$ is a finitely generated $K$-algebra where $K$ is a field, then the \emph{Jacobian ideal} $\cJ_{R/K}$ is defined to be the first nonzero Fitting ideal of $\Omega_{S/K}$, i.e., $\cJ_{R/K}=\sJ_r(R/K)$ if $\sJ_0(R/K)=\cdots=\sJ_{r-1}(R/K)=(0)$ and $\sJ_r(R/K)\neq (0)$.
\end{defn}
\begin{prop}\label{jcb-prop-RankOverReg}
  Let $R$ be an $A$-algebra that is essentially of finite type over $A$ where $A$ is noetherian and universally catenary. Suppose that we have a presentation $R=W^{-1}(T/I)$ where $T=A[\seq{x}{n}]$, $I\subseteq T$ an ideal and $W$ a multiplicatively closed subset of $T$ disjoint from $I$. For a prime ideal $\mfq\in\Spec(R)$, let $Q$ be the preimage of $\mfq$ in $T$ (then, $\mfq\cap W=\emptyset$) and $q:=Q\cap A$. We let $J_{R/A}$ be the Jacobian matrix. Then
  \begin{enumerate}
      \item we have
      \begin{equation}\label{jcb-eq-RkLsHt}
           \rk_{\mfq}(J_{R/A})\leqslant \hght_Q(qT+I)-\hght(q). 
      \end{equation}
     $R$ is smooth over $A$ at $\mfq$ if and only if $A_q\to R_{\mfq}$ is flat and equality holds in \eqref{jcb-eq-RkLsHt}. In this case, we also have
      \[
       \hght_Q(qT+I)-\hght(q)= \rk_{\mfq}(J_{R/A}) = \hght_Q(I)=\mu_Q(I).
      \]
      \item Assume, in addition, that $R$ is reduced, $A$ is regular and $I\cap A=(0)$. Then $\rk_{\mfq}(J_{R/A}) \leqslant \hght_Q(I)$. If we assume furthermore that $R$ is generically smooth over $A$, then $\cJ_{R/A}=\sJ_{n-\bht(I)}(R/A)$.
  \end{enumerate}
 where notation is from \cref{pre-defn-Notation}.
\end{prop}
\begin{proof}
Let $R'=T/I$ and $\mfq'$ be the preimage of $\mfq$ in $R'$. Then $R'_{\mfq'}\cong R_{\mfq}$, and any statement about $R_{\mfq}$ can be proved using the affine $A$-algebra $R'$ and the prime ideal $\mfq'$. Hence, we may replace $R$ by $R'$ and $\mfq$ by $\mfq'$ without affecting anything.

(1): Since everything here is local, we can work over the field $\kappa_q(A)$. Then \cite[7.14]{Kunz1986} shows that $\mu_{\mfq}(\Omega_{R/A})\geqslant \dim_{\mfq}(\kappa_q(A)\otimes_A R)$. Since $\rk_{\mfq}(J_{R/A})=n-\mu_{\mfq}(\Omega_{R/A})$, we have $\rk_{\mfq}(J_{R/A})\leqslant n- \dim_{\mfq}(\kappa_q(A)\otimes_A R)$.

Since $\kappa_q(A)\otimes_A R\cong \kappa_A \otimes_A T/ I(\kappa_A \otimes_A T)$, we have
\begin{align*}
    \dim_{\mfq}\kappa_q(A)\otimes_A R &=\dim_Q \kappa_A \otimes_A T/ I(\kappa_A \otimes_A T)\\
    &=n - \hght_Q(I(\kappa_A \otimes_A T)) \\
    &=n - \hght_Q(I(T/qT)) \\
    &=n-\hght_Q(qT+I) +\hght_Q(qT) \\
    &=n-\hght_Q(qT+I) +\hght(q).
\end{align*}

So $\rk_{\mfq}(J_{R/A})\leqslant \hght_Q(qT+I)-\hght(q)$.

By \cite[8.1]{Kunz1986}, we know that $R$ is smooth over $A$ at $\mfq$ if and only if $A_q\to R_{\mfq}$ is flat and $\mu_q(\Omega_{R/A})\leqslant \dim_{\mfq} R_q/qR_q$. If any of these equivalent conditions is satisfied, then $IT_Q$ is generated by a $T_Q$-regular sequence of length $\dim(T_Q/qT_Q)-\dim (R_{\mfq}/qR_{\mfq})$.

We note that $\dim T_Q/qT_Q = \dim T_Q - \hght(qT_Q) =\hght(Q) - \hght_Q(qT)=\hght(Q)-\hght(q)$. Since $R_{\mfq}/qR_{\mfq} \cong T_Q/(I+pT)T_Q$, we have $\dim R_{\mfq}/qR_{\mfq} = \dim T_Q - \hght((qT+I)T_Q) =\hght(Q) - \hght_Q(qT+I)$. So when $R$ is smooth over $A$ at $\mfq$, we know that $\mu_Q(I)=\hght_Q(I) = \hght_Q(qT+I)-\hght(q)$. In this case, since $IT_Q$ is generated by a $T_Q$-regular sequence, we have $\rk_{\mfq}(J_{R/A}) = \hght_Q(I)$.

(2): By (1), we know that $\rk_{\mfq}(J_{R/A})\leqslant\hght_Q (qT+I)-\hght(q)$.

Since $A$ is regular, so is $T=A[\seq{x}{n}]$ and its localization $T_Q$. By \cref{pre-cor-SerreIntersec}, we have $\hght((qT+I)T_Q)\leqslant \hght(qT_Q)+\hght(IT_Q)$. Note that $qT+I\subseteq Q$, so $\hght((qT+I)T_Q)=\hght_Q(qT+I)$. Since $qT$ is a prime of $T$ contained in $Q$, we have $\hght(qT_Q)=\hght_Q(qT)=\hght(qT)= \hght(q)$. So
\[
\hght_Q(qT+I)\leqslant \hght(q)+\hght_Q(I) \Rightarrow \hght_Q(qT+I)- \hght(q) \leqslant\hght_Q(I)
\]
Hence, we have $\rk_{\mfq}(J_{R/A})=\hght_Q (qT+I)-\hght(q)\leqslant \hght_Q(I)$.

To prove the second statement, we need to show that the maximal rank, i.e., $\bht(I)$ can be achieved. Let $\mfq$ be a minimal prime of $R$ such that $\hght_Q(I)=\bht(I)=\hght(Q)$ where $Q$ is the preimage of $\mfq$ in $T$. Since $I$ is radical, we have $IT_Q=QT_Q$. Let $q:=Q\cap A$. Then 
\[
qA_q = QT_Q\cap A_q = IT_Q\cap A_q = (I\cap A)A_q = (0).
\]
Since $A\to R$ is generically smooth, we know that $A_q\to R_{\mfq}$ is smooth. Hence, in this case, $\rk_{\mfq}(J_{R/A})=\hght_Q(I)=\bht(I)$. Therefore $\sJ_{n-\bht(I)}(R/A)\neq (0)$.

On the other hand, any $i$th Jacobian ideal with $i<n-\bht(I)$ must be zero. If there is some $i_0<n-\bht(I)$ such that $\sJ_{i_0}(R/A)\neq (0)$, then there is some size $n-i_0$ minor nonzero, call it $\Delta$. Since $R$ is reduced, $\Delta$ is not nilpotent. So there is some minimal prime $\mfq'$ of $R$ such that $\Delta\nin\mfq'$. Then $\rk_{\mfq'}(J_{R/A})\geqslant n-i_0 > \bht(I)\geqslant \hght_{Q'}(I)$ where $Q'$ is the preimage of $\mfq'$ in $T$, which violates the first inequality.
\end{proof}

\begin{cor}\label{jcb-cor-RankOverField}
  Let $R$ be a $K$-algebra that is essentially of finite type over $K$, where $K$ is a perfect field. Suppose that we have a presentation $R=W^{-1}T/I$ where $T=K[\seq{x}{n}]$, $I\subseteq T$ an ideal and $W\subseteq T$ a multiplicatively closed subset. For a prime ideal $\mfq\in\Spec(R)$, let $Q$ be the preimage of $\mfq$ in $T$. We let $J_{R/K}$ be the Jacobian matrix. Then
  \begin{enumerate}
  \item  $\rk_{\mfq}(J_{R/K})\leqslant\hght_Q (I)$ and $R$ is smooth over $K$ if and only if the equality holds.
      \item $R$ is regular at $\mfq$ if and only if $\sJ_{\dim T-\hght_Q(I)}(R/K)\not\subseteq \mfq$ where $Q$ is the preimage of $\mfq$ in $T$.
      \item The Jacobian ideal is $\cJ_{R/K}=\sJ_{\dim T-\bht(I)}(R/K)$. If $R$ is equidimensional, then $\bht(I)=\hght(I)$ and $\dim R=\dim T-\hght(I)$. So the Jacobian ideal $\cJ_{R/K}=\sJ_{\dim R}(R/K)$.
      \item Let $\seq{\mfq}{t}$ be the set of all minimal primes of $R$ and let $Q_i$ be the preimage of $\mfq_i$ in $T$ ($1\leqslant i\leqslant t$). Then 
      \[
      \mrm{Sing}(R)= V\left(\prod_{i=1}^t (\sJ_{\dim T-\hght(Q_i)}(R/K) +\mfq_i)\right)= V\left(\bigcap_{i=1}^t (\sJ_{\dim T-\hght(Q_i)}(R/K) +\mfq_i)\right).  
      \]
      If $R$ is equidimensional, then the right-hand side of the above equality simplifies to $V(\cJ_{R/K})$, and we have $\mrm{Sing}(R)= V(\cJ_{R/K})$.
  \end{enumerate}
  where notation is from \cref{pre-defn-Notation}.
\end{cor}
\begin{proof}
For (1), let $A=K$ in \cref{jcb-prop-RankOverReg}. Then $q=0$ and the equality follows.

For (2), $R$ is smooth at $\mfq$ if and only if $\rk_{\mfq}(J_{R/K})=\hght_Q(I)$ if and only if there is a size $\hght_Q(I)$ minor of $J_{R/K}$ outside $\mfq$ if and only if $\sJ_{n-\hght_Q(I)}(R/K)\not\subseteq \mfq$.

For (3), the first statement follows directly from \cref{jcb-prop-RankOverReg}(2) and the fact that any field extension of a perfect field is separable, hence, geometrically regular \cite[5.18, 7.13]{Kunz1986}. The second statement follows from the first one.

For (4), let $\mfp\in\mrm{Sing}(R)$. By (3), we have $\sJ_{n-\hght_P(I)}(R/K)\subseteq \mfp$ where $P$ is the preimage of $\mfp$ in $T$. Suppose that $\mfq_i$ is the minimal prime of $R$ such that $Q_i$ is contained in $P$ and $\hght(Q_i)=\hght_P(I)$. Then $\mfp$ contains both $\mfq_i$ and $\sJ_{n-\hght_P(I)}(R/K)=\sJ_{n-\hght(Q_i)}(R/K)$, which shows that $\mfq_i+\sJ_{n-\hght(Q_i)}(R/K)\subseteq \mfp$. So $\subseteq$ is shown. 

On the other hand, suppose that $\mfp$ is not in the right-hand side of the equality. If $\mfp$ contains a minimal prime $\mfq_i$, then $\mfp$ cannot contain $\sJ_{n-\hght(Q_i)}(R/K)$. This is true for any minimal prime that $\mfp$ contains. So there is some $\mfq_i$ such that $\hght(Q_i)=\hght_P(I)$ and $\mfq_i\subseteq \mfp$. Then $\mfp$ does not contain $\sJ_{n-\hght_P(I)}(R/K)$, so $R$ is regular at $\mfp$, which shows the $\supseteq$ direction.

The last statement about equidimensional rings comes down to the following computation: let $h=\bht(I)=\hght(I)$, then
\begin{align*}
     \bigcap_{i=1}^t (\sJ_{\dim T-\hght(Q_i)}(R/K) +\mfq_i) &=\bigcap_{i=1}^t (\sJ_{n-h}(R/K) +\mfq_i)\\
     &=(\bigcap_{i=1}^t \mfq_i)+\sJ_{n-h}(R/K)\\
     &=(0)+\cJ_{R/K}.
\end{align*}
\end{proof}

\begin{ex}
Let $R=K[x,y,z]/(xz,yz)$ as an example. We have $T=K[X,Y,Z], I=(XZ,YZ)$ with $T\to R$ sending $X\mapsto x,Y\mapsto y,Z\mapsto z$. We know that $\hght(I)=\hght((z))=1,\bht(I)=\hght((x,y))=2$. The Jacobian matrix $J_{R/K}$ is computed to be
 \[
 \begin{pmatrix}
   z & 0 & y \\
   0 & z & x
 \end{pmatrix}.
 \]
 Let $\mfp_0=(x,y)R$ and $P_1=(z)R$ be the minimal primes of $R$ and let $P_0=(X,Y)T$, $P_1=(Z)T$ be their preimages in $T$ respectively. Then we have $\sJ_{3-2}(R/K)=(xz,yz,z^2)R$ and $\sJ_{3-1}(R/K)=(x,y,z)R$.
 
 The Jacobian ideal, by \cref{jcb-cor-RankOverField}(2), is $\sJ_{3-2}(R/K)=(z^2,xz,yz)$, which does not define the singular locus because $(z)\in V(\cJ_{R/K})$ but $R_{(z)}\cong K(Z)$ is regular. By \cref{jcb-cor-RankOverField}(3), we have 
 \[
 (\mfp_0+\sJ_1(R/K))\cap(\mfp_1+\sJ_2(R/K))=(x,y,z^2)\cap (x,y,z)=(x,y,z^2).
 \]
 So the singular locus of $R$ is $V((x,y,z^2))=V((x,y,z))=\set{(x,y,z)}$.
\end{ex}
\begin{rmk}
 In fact, when $R$ defined in \cref{jcb-cor-RankOverField} is not equidimensional, we always have a proper containment
 \[
 \mrm{Sing}(R)\subsetneqq V(\cJ_{R/K}).
 \]
 The containment part is easy as for any $\mfp\in\mrm{Sing}(R)$, by \cref{jcb-cor-RankOverField}(4), we have
 \begin{align*}
     \mfp & \supseteq \bigcap_{i=1}^t (\sJ_{\dim T-\hght(P_i)}(R/K) +\mfp_i) \\
     & \supseteq \bigcap_{i=1}^t (\sJ_{\dim T-\bht(I)}(R/K) +\mfp_i)= \bigcap_{i=1}^t (\cJ_{R/K} +\mfp_i)\\
     & \supseteq  \cJ_{R/K}.
 \end{align*}
 On the other hand, there is some minimal prime $\mfq$ of $R$ such that $\hght(Q)<\bht(I)$ where $Q$ is the preimage of $\mfq$ in $T$. Then since $R$ is regular at $\mfq$, we know that $\mu_{\mfq}(\Omega_{R/K})=\dim T-\hght_Q(I)=n-\hght(Q)$. By \cite[10.6]{Kunz1986}, we have $\sJ_{n-\hght(Q)-1}(R/K)\subseteq \mfq$ and $\sJ_{n-\hght(Q)}(R/K)\not\subseteq \mfq$. Since $\hght(Q)< \bht(I)\Rightarrow \hght(Q)+1\leqslant\bht(I)$, we have $\cJ_{R/K}=\sJ_{n-\bht(I)}(R/K)\subseteq \sJ_{n-\hght(Q)-1}(R/K)$. So $\mfq\in V(\cJ_{R/K})$ but $R$ is regular at $\mfq$.
\end{rmk}

 We generalize \cite[14.13]{Kunz1986} to the affine-analytic case.

\begin{thm}\label{jcb-thm-AbsRankOverField}
 Let $R$ be a reduced affine-analytic $K$-algebra where $K$ is a field. Suppose that we have a presentation $R=T/I$ where $T=K\ps{\seq{x}{n}}[\seq{z}{m}]$. Write $\cA=K\ps{\seq{x}{n}}$. For a prime ideal $\mfq\in\Spec(R)$, let $Q$ be the preimage of $\mfq$ in $T$. We let $J_{R/K}$ be the Jacobian matrix from $\widetilde{\Omega}_{R/K}$. Then the following are equivalent:
  \begin{enumerate}
  \item $R$ is absolutely regular at $\mfq$.
  \item $(\widetilde{\Omega}_{R/K})_{\mfq}$ is a free $R_{\mfq}$-module of rank $n+m-\hght_Q(I)$.
  \item $\mu_{\mfq}(\widetilde{\Omega}_{R/K})\leqslant n+m-\hght_Q(I)$.
  \end{enumerate}
 where notation is from \cref{pre-defn-Notation}.
\end{thm}
\begin{proof}
Let $R'=T/I$, and let $\mfq'$ be the preimage of $\mfq$ in $R'$. Since $R'_{\mfq'}\cong R_{\mfq}$. Note that $R$ may not be reduced. But $R$ is still a finitely generated affine algebra over $A$. Since both \cite[13.15, 13.16]{Kunz1986}, which we use in the following proof, work for the setting whenever $R$ is essentially of finite type over $A$, we will replace $R$ and $\mfq$ by $R'$ and $\mfq'$ respectively and still write them as $R$ and $\mfq$.

Let $q=Q\cap \cA$. We note that $R/\mfq\cong T/Q$. Since $T/Q$ is finitely generated over $\cA/q$, we know that $\mA(\kappa_{\mfq}(R))$ is module-finite over $\cA/q$. Let $F:=\Fr(\mA(\kappa_\mfq(R)))$. Then $\atdeg(F/K)=\dim \cA/q=n-\hght q$.

Since $T$ is finitely generated over $\cA$, by the dimension formula, we have $\hght Q -\hght q = m - \tdeg(\kappa_{\mfq}(R)/F)$ since $\kappa_{\mfq}(R)\cong \Fr(T/Q)$.

So we can write
\begin{align*}
 \dim(R_{\mfq})&+\atdeg(F/K)+\tdeg(\kappa(\mfq)/F)\\
  &=\dim R_{\mfq}+(n-\hght q)+ (m-\hght Q+\hght q)\\
  &=n+m+\dim R_{\mfq} - \hght Q\\
  &=n+m+(\dim T_Q -\hght_Q I) - \dim T_Q\\
  &=n+m-\hght_Q (I).
\end{align*}

To prove the equivalence, we simply make the following modifications to the proof of \cite[14.13]{Kunz1986}
\begin{itemize}
    \item The reference to \cite[14.11]{Kunz1986} is replaced by references to \cref{jcb-prop-ConstFieldExt}.
    \item The reference to \cite[13.15, 13.16]{Kunz1986} for computing the rank is replaced by the calculation above and the rank is replaced by $n+m-\hght_Q(I)$.
\end{itemize}
and the same proof works.
\end{proof}

\begin{cor}\label{jcb-cor-AbsRankOverField}
  Using the notation and the assumption of \cref{jcb-thm-AbsRankOverField}, we have
  \begin{enumerate}
  \item  $\rk_{\mfq}(J_{R/K})\leqslant\hght_Q (I)$ and $R$ is smooth over $K$ if and only if the equality holds.
      \item $R$ is regular at $\mfq$ if and only if $\sJ_{\dim T-\hght_Q(I)}(R/K)\not\subseteq \mfq$ where $Q$ is the preimage of $\mfq$ in $T$.
      \item The Jacobian ideal $\cJ_{R/K}=\sJ_{\dim T-\bht(I)}(R/K)$. If $I$ has pure height $h$, then $\bht(I)=\hght(I)=h$. So the Jacobian ideal $\cJ_{R/K}=\sJ_{\dim T-h}(R/K)$.
      \item If $I$ does not have pure height, then $\mrm{Abs}\mrm{Sing}(R)\subseteq V(\cJ_{R/K})$. If $I$ has pure height, then $\mrm{Abs}\mrm{Sing}(R)= V(\cJ_{R/K})$. Here $\mrm{Abs}\mrm{Sing}$ is the ``absolute singular locus'', which coincides with the (usual) singular locus if $\chr(K)=0$.
  \end{enumerate}
\end{cor}
\begin{proof}
The proof is similar to the proof of \cref{jcb-cor-RankOverField} with the reference to \cref{jcb-prop-RankOverReg} replaced by \cref{jcb-thm-AbsRankOverField}.
\end{proof}


\section{Test Elements in Characteristic \texorpdfstring{$p$}{p}}\label{sn-TEP}
The purpose of this section is to generalize \cite[Corollary 1.5.5]{Huneke1999} to the complete case. We aim to prove the following theorem.

\begin{thm}\label{tep-thm-TEforAnalytic}
  Let $K$ be a field of characteristic $p$ and let $R$ be a $d$-dimensional complete $K$-algebra that is equidimensional and absolutely reduced over $K$. Then the Jacobian ideal $\cJ(R/K)$ is contained in the test ideal of $R$, and remains so after localization and completion.
\end{thm}

Following \cite[Theorem 3.4]{Hochster2002b}, we want to state a more general version of \cite[Corollary 1.5.4]{Huneke1999}.

\begin{prop}\label{noname}
Let $A$ be a regular domain of characteristic $p$. Let $R$ be a module-finite extension of $A$ such that it is torsion-free and generically \'etale over $A$ . Then every element $c$ of $\cJ_{R/A}$ is such that $cR^{1/q}\subseteq A^{1/q}[R]$ for all $q=p^e$, and, in particular, $cR^{\infty}\subseteq A^{\infty}[R]$. Thus, if $c\in \cJ_{R/A
  }\cap R^{\circ}$, it is a completely stable test element.
\end{prop}
\begin{proof}
  If we replace the reference to the usual Lipman-Sathaye theorem with the reference to the ``generalized Lipman-Sathaye Jacobian theorem'' \cite[Theorem 3.1]{Hochster2002}, the same proof in \cite[Corollary 1.5.4]{Huneke1999} works.
\end{proof}

We are ready to prove the main result, \cref{tep-thm-TEforAnalytic}.
\begin{proof}
  By Cohen's structure theorem, we can write $R=L\ps{\seq{X}{n}}/I$ and let $x_i$ be the image of $X_i$ in $R$. Since $R$ is absolutely reduced over $K$, the field extension $K\hookrightarrow L$ is separable. Note that we have a right exact sequence
  $$ \Omega_{L/K}\otimes_L R \to \Omega_{R/K} \to \Omega_{R/L}\to 0 $$
  and $\Omega_{L/K}=0$ due to separability. Thus $\Omega_{R/K} \cong  \Omega_{R/L}$. Since the Jacobian ideal is defined as the first nonzero Fitting ideal of the module of differentials, we conclude that $\cJ(R/K) = \cJ(R/L)$. We will work with $\cJ(R/L)$ from now on.
  
  By \cref{jcb-prop-ConstFieldExt}, we may assume without loss of generality that $L$ is infinite and perfect (Take, for example, the algebraic closure of $L$). Suppose that $I=(\seq{f}{r})$. Then the $(n-d)\times (n-d)$ minors of the Jacobian matrix $\left( \partial f_i / \partial x_j\right)$ generate the Jacobian ideal $\cJ(R/L)$. The calculation of the Jacobian ideal is independent of the choice of coordinates, so we are free to let $\GL_n(L)$ act on the set of variables.
  
  By the discussion \cite[12.14]{Kunz1986} we know that the universal finite differential module $\Omega_{R/L}^1$ is generated by these $\dd x_i$. The total quotient ring $\mQ(R)$ is a finite product of analytically separable field extensions of $L$ by \cite[Theorem 13.10]{Kunz1986}. By the proof of the same theorem, we can modify the generators $\seq{\dd x}{n}$ to get a sequence of elements $\seq{x'}{n}$ such that $\seq{x'}{d}$ form a system of parameters of $R$ and $\seq{\dd x'}{d}$ form a basis for $\Omega_{R/L}^1$.

  Then there is a Zariski dense open subset $U$ of $\GL_n(L)$ such that if we act on the set of variables by an element from $U$ and choose any $d$ of the (new) indeterminates, then the two conditions listed below hold: 
  \begin{enumerate}
  \item The set of $d$ elements form a system of parameters for $R$.
  \item Let $A$ be the complete regular ring generated over $L$ by these $d$ elements. Then $R$ is generically smooth over $A$.
  \end{enumerate}

  By a general position argument we see that there is a Zariski open subset of $\GL_n(L)$ such that the first condition is satisfied. The second condition is satisfied since the differential of any linear combination of $\seq{x'}{n}$ is the same linear combination of their differentials.
  
  Now suppose that a suitable change of coordinates has been made. For any choice of $d$ of these elements, say $\seq{x'}{d}$, let $A$ be the regular local ring $L\ps{\seq{x'}{d}}$. Then $R$ is module-finite over $A$ by the general position argument, and the Jacobian ideal $\cJ_{R/A}$ is generated by the $(n-d)$ size minors of the remaining $n-d$ variables. Since $R$ is equidimensional and reduced, it is likewise torsion-free over $A$. It is generically \'etale because of the general position of the variables. Then \cref{noname} finishes the proof.
\end{proof}

\begin{thm}\label{tep-thm-TEforSemiAnalytic}
Let $R$ be a semianalytic $K$-algebra that is the localization of a reduced equiheight affine-analytic $K$-algebra where $K$ is a field of characteristic $p$. Then the Jacobian ideal $\cJ(R/K)$ is contained in the test ideal of $R$.
\end{thm}
\begin{proof}
We can write $R=W^{-1}T/I$ where $T=K\ps{\seq{x}{n}}[\seq{z}{m}]$ and $I$ is of equiheight in $T$. Suppose that we have a counterexample: there is some element $u\in R$ and some ideal $J\in R$ such that $u\in J^*$ but $\delta u\nin J$ for some $\delta \in \cJ(R/K)$. We can choose a maximal ideal $\mm$ of $R$ such that both $u\in J^*$ and $\delta u\nin J$ still hold in $R_{\mm}$. Then we continue to have these two hold in $\widehat{R_{\mm}}$, the $\mm$-adic completion of $R_{\mm}$.

Since $R_{\mm}$ is reduced, equidimensional and excellent, so is $\widehat{R_{\mm}}$. Let $\mfn$ be the preimage of $\mm$ in $T$. Then $I \widehat{T_{\mfn}}$ is of pure height $\dim \widehat{T_{\mfn}}-\dim \widehat{R_{\mm}}=\dim T_{\mfn}-\dim R_{\mm}$. Since $\Omega_{R_{\mm}/K}=\Omega_{R_{\mm}/K[R^p]}$ is complete. We have $\Omega_{\widehat{R_{\mm}}/K}=\Omega_{R_{\mm}/K}=(\Omega_{R/K})_{\mm}$. Hence the Jacobian ideal $\cJ(R_{\mm}/K)=\cJ(R/K)R_{\mm}$ expands to the Jacobian ideal $\cJ(\widehat{R_{\mm}}/K)$. So $\delta\in \cJ(\widehat{R_{\mm}}/K)$. Since we also have $u\in (J\widehat{R_{\mm}})^*$, we conclude that $\delta u\in J\widehat{R_{\mm}}$ by \cref{tep-thm-TEforAnalytic}, which is a contradiction!
\end{proof}

\section{Definition of Tight Closure in Equal Characteristic 0}\label{sn-DTC}
There are several ways to define tight closure in equal characteristic 0. 
We focus here on $K$-tight closure and on small equational tight closure, which is the case when $K=\QQ$. The tight closure gets larger and the test ideal gets smaller if the field $K$ gets larger. These are the simplest notions to define and there does not appear to be much motivation to use more complicated notions.

In this section, we briefly introduce the definition of $K$-tight closure in equal characteristic 0. We usually omit the reference to $K$ in the definition. We start with affine $K$-algebras, then we pass to noetherian $K$-algebras. In the case $K=\QQ$, this is called the small equational tight closure.

\subsection{Tight closure for affine algebras over fields of characteristic 0}
Let $R$ be an affine $K$-algebra where $K$ is a field of characteristic $0$. Let $N\subseteq M$ be finitely generated $R$-modules and $u\in M$ an element in the module. We want to ``descend'' the data over $R$, i.e., the quintuple $(K,R,N,M,u)$, to some finitely generated $\ZZ$-subalgebra $\cA$ of $K$. Roughly speaking, the \emph{descent data} for the quintuple from $R$ to $\cA$ is also a quintuple $(\cA,R_{\cA},N_{\cA},M_{\cA},u_{\cA})$ such that when tensored with $K$ over $\cA$, we recover the original quintuple. The formal definition below is taken from \cite[Descent data 2.1.2]{Huneke1999}.
\begin{defn}
  Let a quintuple $(K,R,M,N,u)$ be defined as above. By \emph{descent data} for this quintuple, we mean a quintuple $(\cA,R_{\cA},N_{\cA},M_{\cA},u_{\cA})$ satisfying the following conditions:
  \begin{enumerate}
  \item $\cA$ is a finitely generated $\ZZ$-subalgebra of $K$.
  \item $R_{\cA}$ is a finitely generated $\cA$-subalgebra of $R$ such that the inclusion $R_{\cA}\subseteq R$ induces an isomorphism of $R_K$ with $R$. Moreover, $R_{\cA}$ is $\cA$-free.
  \item $M_{\cA},N_{\cA}$ are finitely generated $\cA$-submodules of $M,N$ respectively such that $N_{\cA}\subseteq M_{\cA}$ and all of the modules $M_{\cA},N_{\cA},M_{\cA}/N_{\cA}$ are $\cA$-free. Moreover, the diagram below
    \[
      \xymatrix{
        N_{\cA}\ar@{^{(}->}[r]\ar@{^{(}->}[d]  & M_{\cA}\ar@{^{(}->}[d] \\
        N\ar@{^{(}->}[r] & M
      } \quad\Rightarrow\quad
      \xymatrix{
        N_{\cA}\otimes_{\cA} K \ar@{^{(}->}[r]\ar[d]^{\cong}  & M_{\cA}\otimes_{\cA} K\ar[d]^{\cong} \\
        N\ar@{^{(}->}[r] & M
      } 
    \]
    as $R$-modules.
  \item The element $u\in M$ is in $M_{\cA}$ and $u=u_{\cA}$.
  \end{enumerate}
\end{defn}
The most important fact is that descent data do exist \cite[Discussion: the existence of descent data. 2.1.3]{Huneke1999}, and in fact there are a lot of them. We actually have $R=\varinjlim_B R_B$ where $B$ runs through all finitely generated $\ZZ$-subalgebras with $\cA\subseteq B\subseteq K$ and $R_B$ is the descent of $R$. Similarly we have $M = \varinjlim_B M_B $ and $N = \varinjlim_B N_B$ \cite[Proposition 2.1.9]{Huneke1999}. Let us give an example to illustrate this definition.
\begin{ex}
  Let $R=\QQ[x,y,z]/(x^2/2+y^3/3+z^5/5)$ and $K=\QQ$. Then we can take $\cA$ to be $\ZZ[1/2,1/3,1/5]$. Therefore $x^2/2+y^3/3+z^5/5$ makes sense in $\cA[x,y,z]$ and we can form $R_{\cA}=\cA[x,y,z]/(x^2/2+y^3/3+z^5/5)$. In fact, take $B$ to be any finitely generated $\cA$-algebra such that $\cA\subseteq B\subseteq \QQ$, then $R_B:=R_{\cA}\otimes_{\cA} B$ will be a descent of $R$ to $B$.
\end{ex}

We give the definition of the tight closure notion on a finitely generated $\ZZ$-algebra $\cA$ below. 

\begin{conv}\cite[2.2.2]{Huneke1999}\label{conv-almostall}
For a finitely generated $\ZZ$-algebra $A$, a property $\mathsf{P}$ holds for almost all $\mu\in\mrm{Max}\Spec(\cA)$ if there is some open dense subset $U$ of $\mrm{Max}\Spec(\cA)$ such that $\mathsf{P}$ holds for all $\mu\in U$. Let $\mathsf{Q}$ be a class of rings (e.g., all domains). By ``for almost all rings in $\mathsf{Q}$ that $A$ maps to'' we mean that ``there is an element $a\in A$ such that for all rings in $\mathsf{Q}$ that $A_a$ maps to.''
\end{conv}

\begin{defn}[Affine case]\label{dtc-defn-AffineCase}
  Let $\cA$ be a finitely generated $\ZZ$-algebra. Let $M_{\cA}$ be an $\cA$-module and $N_{\cA}\subseteq M_{\cA}$ a submodule. We say that $u_{\cA}\in M_{\cA}$ is in $\tca{(N_{\cA})}_{M_{\cA}}$ if for almost all (\cref{conv-almostall}) $\mu\in\mrm{Max}\Spec(\cA)$, $u_{\kappa}\in \ang{N_{\kappa}}^{*}_{M_{\kappa}}$ where $\kappa=\cA/\mu$.
\end{defn}

Based on the affine case, we define the tight closure for a finitely generated $K$-algebra $R$ as follows:
\begin{defn}[Finitely generated $K$-algebra]\label{dtc-defn-FGKAlgebraCase}
  Let $R$ be a finitely generated $K$-algebra and let $N\subseteq M$ be $R$-modules. We say that $u\in M$ is in the tight closure of $\tck{N}_M$ if there exists descent data $(\cA,R_{\cA},M_{\cA},N_{\cA},u_{\cA})$ for $(K,R,M,N,u)$ such that $u_{\cA}\in \tca{(N_{\cA})}_{M_{\cA}}$ over $R_{\cA}$ in the sense of \cref{dtc-defn-AffineCase}.
\end{defn}

\subsection{Test elements in the affine case}
We want to generalize both \cite[Theorem 2.4.9, Corollary 2.4.10]{Huneke1999} to the non-domain case.
\begin{prop}
Let $\cA$ be a finitely generated $\ZZ$-domain with fraction field $\cF$, and let $R_{\cA}$ containing $\cA$ be a finitely generated $\cA$-algebra. Suppose that $R_{\cA}$ is module-finite over a regular ring $\cT_{\cA}$ and that $R_{\cF}$ is geometrically reduced. Then the nonzero elements in the Jacobian ideal $\cJ(R_{\cA}/\cT_{\cA})$ are universal test elements for $\cA\to R_{\cA}$ in the sense of \cite[Definition 2.4.2]{Huneke1999}.
\end{prop}
\begin{proof}
  Localize at one element of $\cA^{\circ}$ so that $\cA$ is regular, $\cA\to \cT_{\cA}$ is smooth and also so that $R_{\cA}$ is $\cA$-free. For almost all (\cref{conv-almostall}) fields $\cL$ to which $\cA$ maps, $R_{\cL}$ is geometrically reduced \cite[(2.3.6)b]{Huneke1999}. If follows that for almost all (\cref{conv-almostall}) regular domains $\Lambda$, $R_{\Lambda}$ is geometrically reduced and the extension $\cT_{\Lambda}\subseteq R_{\Lambda}$ is module-finite. Now the result is immediate from \cref{noname}
\end{proof}
\begin{cor}\label{corcorcor}
Let $\cA$ be a finitely generated $\ZZ$-domain with fraction field $\cF$, and let $R_{\cA}$ containing $\cA$ be a finitely generated $\cA$-algebra. Suppose that $R_{\cF}$ is geometrically reduced and $d$-dimensional. Then the nonzero elements in the Jacobian ideal $\cJ(R_{\cA}/\cA)$ are universal test elements for $\cA\to R_{\cA}$ in the sense of \cite[Definition 2.4.2]{Huneke1999}.
\end{cor}

\begin{lem}\label{lemmaabove}
Let $\cA$ be a finitely generated $\ZZ$-domain and let $R_{\cA}$ containing $\cA$ be a finitely generated $\cA$-algebra. Suppose that $I_{\cA}\subseteq R_{\cA}$ is an ideal and $u_{\cA}\in R_{\cA}$ is an element. If $u_{\kappa}\in I_{\cA}R_{\kappa}$ for almost all (\cref{conv-almostall}) $\kappa\in\mrm{Max}\Spec(A)$, then $u_{\cA}\in I_{\cA}$.
\end{lem}
\begin{proof}
Assume for contradiction that $u_{\cA}\nin I_{\cA}$. Consider the $R_{\cA}$-module $(I_{\cA}+u_{\cA})/I_{\cA}$. By assumption it is a finitely generated nonzero module. By \cite[Lemma 8.1]{Hochster1974} we can localize at one element of $\cA$ to make it free. 
Hence its rank can be checked by base change to any $\cA_{\kappa}$. Then we conclude that it has rank 0, i.e., it is a zero module. So we have $u_{\cA}\in I_{\cA}$.
\end{proof}

By \cref{dtc-conv-TestIdeal}, in order to show that some element is a test element, we will show that it multiplies the tight closure of any ideal in the ring back to the ideal.

\begin{thm}\label{HH99}
  Let $R$ be a reduced finitely generated equidimensional $K$-algebra of Krull dimension $d$. Then $\cJ(R/K)$ is contained in the test ideal.
\end{thm}
\begin{proof}
  We write $R=K[\seq{x}{n}]/(\seq{f}{r})$. Then $(\seq{f}{r})$ has pure height $n-d$, and the Jacobian ideal $\cJ(R/K)$ is generated by the size $(n-d)$ minors of the Jacobian matrix. Let $\delta$ be one of the minors.
  
  Let $I$ be an ideal of $R$ and let $u\in \tck{I}$. Let $A$ be the finitely generated $\ZZ$-subalgebra of $K$ such that all polynomials of $\seq{f}{r}$ and all size $(n-d)$ minors of the Jacobian matrix are defined in $A[\seq{x}{n}]$. Localizing at finitely many nonzero elements (equivalently, one nonzero element) of $A$, we can make $(A,R_A=A[\seq{x}{n}]/(\seq{f}{r}),I_A,u_A)$ a descent of $(K,R,I,u)$. We will write $\delta\in R_A$ by abusing notation. Then clearly $\delta\in \cJ(R_A/A)$ and remains so after any base change $A\to B$ where $A\subseteq B\subseteq K$.
  
  By the definition of tight closure in characteristic zero, there exists a finitely generated $\ZZ$-subalgebra $A_0 \subseteq K$ and a descent $(A_0, R_{A_0}, I_{A_0}, u_{A_0})$ such that $u_{A_0}\in \tca[A_0]{I_{A_0}}$. By enlarging $A$ to include the generators of $A_0$ (and localizing further if necessary, as in \cite[2.1.6]{Huneke1999}), we may assume without loss of generality that $A_0 \subseteq A$. Over this enlarged ring $A$, both the Jacobian conditions and the tight closure relation $u_A\in \tca[A]{I_A}$ are simultaneously satisfied. Hence, for almost all $\mu\in\mrm{Max}\Spec(A)$, \cref{corcorcor} implies that the image of $\delta$ in $R_{\kappa}$ is a test element. Consequently, we have $\delta u_{\kappa} \in I_AR_{\kappa}$. By \cref{lemmaabove}, we conclude that $\delta u_A\in I_A$, which implies that $\delta u\in I$ in $R$. Since $\delta$ and $u$ are arbitrarily chosen, we conclude that $\cJ(R/K) \tck{I} \subseteq I$ for any ideal $I\subseteq R$.
\end{proof}

\subsection{Affine progenitors}
Let $K$ be a field and let $S$ be a noetherian $K$-algebra. Note that $S$ may not necessarily be finitely generated over $K$. Let $N\subseteq M$ be finitely generated $S$-modules, $\underline{u}$ a finite sequence of elements of $M$. We have the following definition of an affine progenitor.
\begin{defn}\cite[Definition 3.1.1]{Huneke1999}\label{dtc-defn-AffineProgenitor}
  By an \emph{affine progenitor} for $(S,M,N,\underline{u})$ we shall means a septuple $\cM=(R,M_R,N_R,\underline{u}_R,h,\beta,\eta_R)$ where
  \begin{itemize}
  \item $R$ is a finitely generated $K$-algebra.
  \item $h:R\to S$ a $K$-homomorphism.
  \item $M_R$ is a finitely generated $R$-module with an $R$-linear map $\beta:M_R\to M$ such that the induced map $\beta_{*}:S\otimes_R M_R\to M$ is an isomorphism.
  \item $\underline{u}_R$ is a finite sequence of elements of $M_R$ such that $\beta_{*}$ maps $\underline{u}_R$ to $\underline{u}$.
  \item $\eta_R$ is an $R$-linear map from $N_R$ to $M_R$ and the induced map $N_S\to M_S\to M$ maps $N_S$ onto $N$.
  \end{itemize}
  We refer to $R$ as the \emph{base ring} of the affine progenitor.
\end{defn}
The data of the affine progenitor is captured by the following diagram.
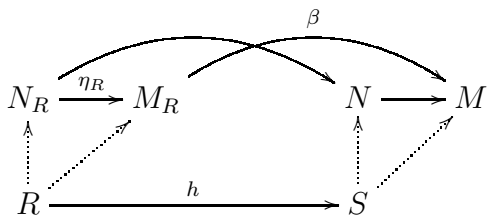
\[
  \xymatrix{
    N_R\ar[r]^{\eta_R}\ar@/^2em/[rrr]  & M_R\ar@/^2em/[rrr]^{\beta} & & N\ar[r] & M\\
    R\ar@{..>}[u]\ar@{..>}[ur]\ar[rrr]^h & & & S\ar@{..>}[u]\ar@{..>}[ur] & 
  }
\]
where the dashed arrow $R\dashedrightarrow M$ means $M$ is an $R$-module. Note that we do not require $\eta_R$ to be injective, nor do we require that $N_S$ be isomorphic to $N$. Also we do not require $R$ to be a subring of $S$.

\subsection{General noetherian \texorpdfstring{$K$}{K}-algebras}
We are ready to define a notion of tight closure, called direct tight closure, for a general noetherian $K$-algebra $S$.
\begin{defn}
  Let $S$ be a noetherian $K$-algebra and $N\subseteq M$ be finitely generated $S$-modules. Let $u\in M$. Then we say that $u$ is in the \emph{direct $K$-tight closure} $\dtc{N}$ of $N$ in $M$ if there exists an affine progenitor $(R,M_R,N_R,u_R)$ (\cref{dtc-defn-AffineProgenitor}) for $(S,M,N,u)$ such that $u_R\in \tck{\ang{N_R}}_{M_R}$ as in \cref{dtc-defn-FGKAlgebraCase}.
\end{defn}

We also have a notion of tight closure called ``formal $K$-tight closure,'' defined below.
\begin{defn}
  Let $R$ be a noetherian $K$-algebra. By a complete local domain $B$ (at prime $P$) of $R$ we mean $\widehat{R}_P$ modulo a minimal prime, where $P\subseteq R$ is a prime ideal. We say that $u$ is in the \emph{formal $K$-tight closure} $\ftc{N}$ of $N$ in $M$ if for every complete local domain $B$ of $R$, $1\otimes u$ is in the direct $K$-tight closure of $\ang{B\otimes_R N}$ in $B\otimes_R M$.  
\end{defn}

We will use the definition of formal $K$-tight closure for the actual definition for all cases. This will not cause any conflicts as we have the following remarkable result.
\begin{thm}\cite[Theorem 3.4.1]{Huneke1999}
  Let $S$ be a locally excellent noetherian algebra over a field $K$ of characteristic $0$. Let $N\subseteq M$ be finitely generated $S$-modules. Then the following three conditions on an element $u\in M$ are equivalent:
  \begin{enumerate}
  \item $u\in \dtc{N}_M$.
  \item For every maximal ideal $\mm$ of $S$, if $C=\widehat{S_{\mm}}$ then $u_C\in \dtc{\ang{N_C}}_{M_C}$.
  \item $u\in \ftc{N}_M$.
  \end{enumerate}
  In particular, we have $\dtc{N}_M=\ftc{N}_M$.
\end{thm}

Consequently, for all three cases we have the following corollary.
\begin{cor}\cite[Corollary 3.4.2]{Huneke1999}
  Let $R$ be a finitely generated algebra over a field $K$ of characteristic zero. Let $N\subseteq M$ be finitely generated $R$-modules. Then $\tck{N}_M=\dtc{N}_M=\ftc{N}_M$.
\end{cor}

Finally, we are ready to give the definition of (small) equational tight closure in equal characteristic 0.
\begin{defn}\cite[Definition 3.4.3]{Huneke1999}
  Let $R$ be a noetherian $K$-algebra, where $K$ is a field of characteristic zero, and let $N\subseteq M$ be finitely generated $R$-modules.
  \begin{enumerate}
  \item We define the $K$-tight closure $\tck{N}$ of $N$ in $M$ to be the formal $K$-tight closure of $N$ in $M$.
  \item Every noetherian ring $R$ of equal characteristic 0 is (uniquely) a $\QQ$-algebra. When $K=\QQ$ we shall refer to the direct $\QQ$-tight closure of $N$ in $M$ as the \emph{direct equational tight closure} of $N$ in $M$, and denote it $\deqtc{N}_M$. we shall refer to the $\QQ$-tight closure of $N$ in $M$ as the \emph{equational tight closure} of $N$ in $M$, and denote it $\eqtc{N}_M$.
  \end{enumerate}
\end{defn}


\section{Test Elements in Characteristic 0}\label{sn-TEZA}
We aim to prove \cref{teza-thm-TEforAnalytic}. Before that, we have to make great use of the Artin-Rotthaus theorem (\cref{teza-thm-FlatGeoRegFiber}) to discuss a series of descent results ((A1) - (A11)).
\subsection{Descent data} We need the following Artin-Rotthaus theorem \cite{Artin1988}.
\begin{thm}[Artin-Rotthaus]\label{teza-thm-FlatGeoRegFiber}
  Let $K$ be a field. Then the power series ring $K\ps{\seq{x}{n}}$ is a direct limit of smooth $K[\seq{x}{n}]$-algebras.
\end{thm}

The Artin-Rotthaus theorem is also a consequence of N\'eron-Popescu desingularization \cite[Theorem 1.1]{Swan1998}.
\begin{thm}
  Let $f:A\to B$ be a ring homomorphism.  Then $f$ is flat with geometrically regular fibers if and only if $B$ is a filtered colimit of smooth $A$-algebras.
\end{thm}

We first note that maps in the Artin-Rotthaus theorem are not necessarily injective. Furthermore, when $n > 0$, the Krull dimension of the algebras occurring in the direct limit must become arbitrarily large for large indices. This is because the fraction field of $K\ps{\seq{x}{n}}$ has infinite transcendence degree over $K(\seq{x}{n})$. If an algebra in the direct limit system has $m$ elements whose images in $K\ps{\seq{x}{n}}$ are algebraically independent over $K(\seq{x}{n})$, then the Krull dimension of that algebra must be at least $m+n$. Because of this dimension blow-up, we will primarily study the height of ideals generated by certain elements in these algebras, rather than the Krull dimension of the quotient rings.

We will be able to preserve many properties while passing to a larger algebra, i.e., these properties will hold for all algebras occurring after a certain index $\nu$.

Let $R$ be a reduced, equidimensional complete local ring of dimension $d$ with coefficients field $K$, i.e., $R=K\ps{\seq{x}{n}}/(\seq{f}{m})$. We write $A=K[\seq{x}{n}]$. Then $\widehat{A}=K\ps{\seq{x}{n}}$, and $f_i\in \widehat{A}$ for each $i$.
\begin{enumerate}[left=0pt,label=(A\arabic*)]
\item (Eventually injective) Note that for any element $a\in \widehat{A}$, there is some $\nu$ and some $\widetilde{a}\in A_{\nu}$ mapping to $a$. Note that $A_{\nu}\to \widehat{A}$ may not be injective. But the kernel is a finitely generated ideal and maps to zero in $\widehat{A}$, hence it must be zero when we map to a large enough algebra $A_{\mu}$. Moreover, the image of $A_{\nu}$ in $A_{\mu}$ maps injectively to $\widehat{A}$, and remains so mapping to any $A_{\lambda}$ where $\lambda\geqslant \mu$. Therefore, we will denote the image of $\widetilde{a}$ in these $A_{\lambda}$ by $a$ directly.\label{AR1}
\item (Descent of elements) We shall write ``for all $\mu\gg\nu$'' to mean that ``there exists some $\lambda> \nu$ and for all $\mu\geqslant \lambda$.'' By \ref{AR1}, for any element $a\in \widehat{A}$, there exists some $\nu$ such that $a\in A_{\mu}$ for all $\mu\gg\nu$. Hence, we can say that there exists some $\mu$ such that $a\in A_{\mu}$. Of course, for any $A_{\lambda}$ where $\lambda\geqslant\mu$, we have $a\in A_{\lambda}$.\label{AR2}
\item (Notation) We shall frequently use the following notation:
  \begin{itemize}
  \item We write $\mm$ for the maximal ideal of $\widehat{A}$ and let $\mm_{\nu}$ denote the contraction of $\mm$ in $A_{\nu}$. Since $A_{\nu}/\mm_{\nu}\hookrightarrow \widehat{A}/\mm = K$, and so must be $K$. Hence $\mm_{\nu}$ is a maximal ideal of $A_{\nu}$.
  \item Note that any element $a$ in $A_{\nu}-\mm_{\nu}$ maps to an element in $\widehat{A}-\mm$. So its image is necessarily a unit, and we actually have $(A_{\nu})_a\to \widehat{A}$. Also, we have such maps for $(A_{\nu})_a\to A_{\mu}$ for all $\mu\gg\nu$. Moreover, we can do this for finitely many such elements by localizing at their product.
  \item By the bullet point above, the localizations we obtained above are cofinal with the direct system. So we can always assume for each $\mu\gg\nu$, a localization is made, if needed, and we shall indicate this by ``for all $\mu\gg_{\loc}\nu$.''
  \end{itemize}\label{AR3}
\item (Descent of ideals) For any ideal $\mfa\subseteq \widehat{A}$, if $\mfa$ is generated by $\seq{a}{k}$, then there exists a $\nu$ such that all of the $a_i$ are in $A_{\nu}$. Therefore we have $\mfa_{\mu}:=(\seq{a}{k})A_{\mu}$ for all $\mu\geqslant\nu$. If in addition, we have an element $u\in \mfa\subseteq \widehat{A}$, we can write $u=z_1a_1+\cdots+z_ka_k$ for some $\seq{z}{k}\in \widehat{A}$. By choosing a larger $\mu$ we have $u,\seq{z}{k}\in A_{\mu}$. Then the element $u-\sum_{i=1}^k z_ia_i\in A_{\mu}$ maps to zero in $\widehat{A}$. By passing to a larger $\mu$ we may assume that this is honestly zero. Therefore for all $\mu\gg\nu$, we have $u\in\mfa_{\mu}$.\label{AR4}
\item (Description of the maximal ideal) Any element in $\mm_{\nu}$ is in $(\seq{x}{n})A_{\mu}$ for all $\mu\gg \nu$. Let $\widetilde{a}\in \mm_{\nu}$ be an element. This is trivial if $\widetilde{a}$ maps to zero. If $\widetilde{a}$ maps to a nonzero element $a$ in $\mm=(\seq{x}{n})\widehat{A}$, by \ref{AR4}, there exists $\mu$ such that $a\in (\seq{x}{n})A_{\mu}$.\label{AR5}
\item (Descent of radicals) Let $\seq{u}{n}\in \widehat{A}$ be a system of parameters. Then by \ref{AR2} and \ref{AR5}, there exists some $\nu$ such that $\seq{u}{n}\in A_{\mu}$ and $(\seq{u}{n})A_{\mu}\subseteq (\seq{x}{n})A_{\mu}$ for all $\mu\geqslant \nu$. Since each $x_i$ has some power in $ (\seq{u}{n})\widehat{A}$, we choose a larger $\mu$ such that each $x_i$ also has a power in $(\seq{u}{n})A_{\mu}$. Then this implies that $(\seq{u}{n})A_{\mu}$ and $(\seq{x}{n})A_{\mu}$ have the same radical.\label{AR6}
\end{enumerate}

The next proposition says that we can descend regular sequences in $\widehat{A}$.
\begin{prop}\label{teza-AR-RegSeqAlmostAll}
  If elements $\seq{u}{k}$ in $\widehat{A}$ form a regular sequence, then there exists $\nu$ such that they form a regular sequence in $A_{\mu}$ for all $\mu\gg_{\loc}\nu$.
\end{prop}
\begin{proof}
  We can extend $\seq{u}{k}$ to a full system of parameters $\seq{u}{n}\in \widehat{A}$. If we can show that $\seq{u}{n}$ form a regular sequence in $A_{\mu}$ for all $\mu\gg_{\loc}\nu$, then the conclusion follows immediately.
  
  By \ref{AR6} we can find $\mu$ such that $(\seq{u}{n})A_{\mu}\subseteq(\seq{x}{n})A_{\mu}$ and they have the same radical. By the construction of $A_{\mu}$, we know that $\seq{x}{n}$ form a regular sequence on $A_{\mu}$, which implies that the Koszul homology $\mH_i(\seq{x}{n};A_{\mu})$ vanishes for each $i$. Since they have the same radical, we also have $\mH_i(\seq{u}{n};A_{\mu})$ vanish. Then $\seq{u}{n}$ form a regular sequence on $(A_{\mu})_{\mm_{\mu}}$.

  The modules $\dfrac{(\seq{u}{h}):_{A_{\mu}}u_{h+1}}{(\seq{u}{h})A_{\mu}}$ are finitely generated and become zero once we localize at $\mm_{\mu}$. It is clear that we can localize at one element $r$ to make all these modules zero. Thus $\seq{u}{n}$ form a regular sequence on $(A_{\mu})_r$.
\end{proof}

Next we observe that we can descend ideals while preserving their heights, see \ref{AR7}; if the ideal has pure height, we can preserve that, see \ref{AR11}. For this purpose, we need the following fact.
\begin{fact}\cite[Facts 2.3.7]{Huneke1999}\label{teza-AR-Height}
  Let $R$ be a noetherian ring and $I$ an ideal of $R$. 
  \begin{enumerate}
  \item If $I$ is proper then $I$ has height at least $h$ if and only if there is a sequence of elements $\seq{x}{h}$ in $I$ such that for all $i$, $0\leqslant i\leqslant h-1$, $x_{i+1}$ is not in any minimal prime of $(\seq{x}{i})R$.
  \item $I$ has height at most $h$ if and only if there exists a proper ideal $J$ containing $I$ and an element $y$ of $R$ not a zerodivisor on $J$ such that $yJ\subseteq \sqrt{I}$ and $yJ$ is contained in the radical of an ideal generated by at most $h$ elements.
  \end{enumerate}
\end{fact}
\begin{proof}
  (1) Since $I$ is not in any minimal prime of $(0)R$, we can choose $x_1\in I$ avoiding all minimal primes of $(0)R$. Suppose that $\seq{x}{i}$ are chosen. Since $\hght(I)\geqslant h >i$, it is not contained in any minimal primes of $(\seq{x}{i})R$. Hence, it is not contained in the union of all these minimal primes. So we can choose $x_{i+1}\in I$ avoiding all minimal primes.
  
  (2) If $I\subseteq J$, and $y$ is a nonzerodivisor on $J$, then $y$ avoids all associated primes of $I$ and $J$. So we have $\hght(I)=\hght(IR_y)$ and $\hght(J)=\hght(JR_y)$. Note that in $R_y$, we have $JR_y\subseteq \sqrt{I}R_y\subseteq \sqrt{JR_y}$. Since $JR_y$ is contained in the radical of an ideal generated by at most $h$ elements, it has height at most $h$. So $\hght(I)=\hght(IR_y)=\hght(\sqrt{I}R_y)\leqslant\hght(JR_y)\leqslant h$. Now we assume that $I$ has height at most $h$. Let $P$ be one of the minimal primes of $I$ such that $\hght(P)=\hght(I)\leqslant h$. Then $IR_P$ is $PR_P$-primary. So $PR_P=\sqrt{I}R_P$. We also know that $PR_P$ is the radical of at most $h$ elements (a system of parameters) in $R_P$. By looking at the generators, we can localize at one element $y\in R-P$ such that $PR_y\subseteq \sqrt{I}R_y$ and $PR_y$ is contained in the radical of an ideal generated by at most $h$ elements. If we raise $y$ to a power if necessary and let $J=P$, then we have $yJ\subseteq \sqrt{I}$ and $yJ$ is contained in the radical of an ideal generated by at most $h$ elements.
\end{proof}
\begin{enumerate}[resume, left=0pt, label=(A\arabic*)]
\item (Preserving height while descending) For any ideal $\mfa=(\seq{a}{n})\widehat{A}$ of height $h$, we have $\mfa_{\mu}$ for all $\mu\gg\nu$. On the one hand, let $\seq{x}{h}$ be a maximal regular sequence in $\mfa$. Then we have $(\seq{x}{h})\subseteq(\seq{a}{n})$ in $A_{\mu}$, and $\seq{x}{h}$ continue to be a regular sequence for all $\mu\gg_{\loc} \nu$ by \cref{teza-AR-RegSeqAlmostAll}. So we have $\hght(\mfa_{\mu})\geqslant h$ for all $\mu\gg_{\loc} \nu$. On the other hand, by \cref{teza-AR-Height}, there is some ideal $J$ and an element $y$ in $\widehat{A}$ such that
  \begin{itemize}
  \item $I\subseteq J$.
  \item $yJ\subseteq \sqrt{I}$.
  \item There is some power $l$ such that $(yJ)^l\subseteq K$ where $K$ is generated by $h$ elements.
  \item $y$ is a nonzerodivisor on $J$.
  \end{itemize}
  All these except the last bullet point can be achieved using \ref{AR4}. The last one is done by \cref{teza-AR-RegSeqAlmostAll}. Hence for all $\mu\gg_{\loc}\nu$, we also have $\hght(a_{\mu})\geqslant h$. \label{AR7}
\end{enumerate}

Next we want to prove some results about descending modules and exact sequences by descending their presentations.

\begin{enumerate}[resume, left=0pt, label=(A\arabic*)]
\item (Descent of finitely generated modules and maps in-between) Let $M$ be a finitely generated $\widehat{A}$-module. Since $\widehat{A}$ is noetherian, we can write $M$ as the cokernel of a matrix $\alpha$. For large enough $\nu$, we have all entries of $\alpha$ in $A_{\nu}$, so we can form $M_{\nu}:=\Coker(\alpha_{\nu})$. For a map $f:M\to N$ between $\widehat{A}$-modules $M,N$, we have lifting of maps
  \[
    \xymatrix{
      F_1\ar@{>>}[r]\ar@{.>}[d]^h & K\ar@{^{(}->}[r]\ar@{.>}[d]^{g|_K} & F_0\ar@{>>}[r]\ar@{.>}[d]^g & M\ar[d]^f \\
      F_1'\ar@{>>}[r] & K'\ar@{^{(}->}[r] & F_0'\ar@{>>}[r] & N
    }
  \]
  where all $F_1,F_0,F_1',F_0'$ are free $\widehat{A}$-modules. So we have commutative diagram
  \[
    \xymatrix{
      F_1\ar[r]^{\alpha}\ar[d]^h & F_0\ar[d]^g \\
      F_1'\ar[r]^{\beta} & F_0'
    }.
  \]
  Choose $\nu$ large enough such that all matrices makes sense and the corresponding diagram commutes. Then we get a map from $M_{\nu}=\Coker(\alpha_{\nu})$ to $N_{\nu}=\Coker(\beta_{\nu})$, which recovers $M\to N$ once we tensor with $\widehat{A}$.\label{AR8}
\item (Descent of finite free resolutions) A finite free resolution of a module $\Coker(\alpha)$ over $\widehat{A}$ descends to a finite free resolution over some $A_{\nu}$.\label{AR9}
  \begin{proof}
    We need the Buchsbaum-Eisenbud criterion for acyclicity \cite{Buchsbaum1973}. Let $F_{\bullet}=0\to F_n\to \cdots\to F_0$ be a complex of finite free $\widehat{A}$-modules and let $\varphi_i:F_i\to F_{i-1}$ be the maps. We can descend each map $\varphi_i$ (as a matrix) to a large enough $A_{\nu}$. Call the descended complex $F_{\bullet}'$ and the corresponding maps $\varphi_i'$. Then after passing to some $\mu\geqslant \nu$, we have
    \begin{itemize}
    \item The compositions of consecutive maps are zero.
    \item The determinantal rank of $\varphi_i'$ is the same as over $\widehat{A}$. So $\rk(\varphi_{i+1}')+\rk(\varphi_i')=\rk(F_i')$.
    \item The ideal generated by the rank size minors at $i$th spot $I_{\rk(\varphi_i')}(\varphi_i')$ is either the whole ring or contains a regular sequence of length $i$.
    \end{itemize}
    The first and the second bullet points are achieved by \ref{AR2} and the facts that $\varphi_{i+1}\circ\varphi_i=0,I_{\rk(\varphi_i)+1}(\varphi_i)=0$ over $\widehat{A}$. For the third bullet point, the proof splits into two cases:
    \begin{itemize}
        \item If $I_{\rk(\varphi_i)}(\varphi_i)=\widehat{A}$, then $I_{\rk(\varphi_i')}(\varphi_i')$ contains some element that maps to a unit in $\widehat{A}$. Hence this element will become invertible after a suitably large index $\nu$.
        \item If $I_{\rk(\varphi_i)}(\varphi_i)$ contains a regular sequence of length at least $i$ in $\widehat{A}$, by \cref{teza-AR-RegSeqAlmostAll} we know that this holds in a suitable localization of $A_{\mu}$ for large enough $\mu$.
    \end{itemize}
  So acyclicity follows for $\mu\gg_{\loc}\nu$.
  \end{proof}
\item (Descent of short exact sequences) A short exact sequence of finitely generated modules over $\widehat{A}$ descends to a short exact sequence over some $A_{\nu}$.\label{AR10}
  \begin{proof}
    We take finite free resolutions of the first and the third modules of the sequence, and fill in maps for the direct sums of the free modules for a given degree to give a free resolution of the middle module in the sequence. Then we can descend the whole resolution and the maps between them by \ref{AR9}.
  \end{proof}
\end{enumerate}

\begin{lem}\label{teza-AR-PureHeight}
  Let $T$ be a Gorenstein local ring. A finitely generated $T$-module $M$ has pure dimension $\dim(T)-h$ if and only if it embeds in a finite direct sum of modules of the form $T/(\seq{x}{h})T$, where each $\seq{x}{h}$ is a regular sequence in $T$.
\end{lem}
\begin{proof}
  If $M$ embeds into such a direct sum, then $M$ clearly has pure height $h$. Assume that $M$ has pure height $h$, and let $\seq{P}{n}$ be the set of associated primes of $M$. Then each $P_i$ has height $h$ for $1\leqslant i\leqslant n$. Consider the natural map $M\to \oplus_{i=1}^n M_{P_i}$. The kernel of this map consists of elements killed by $W=T-\cup_{i=1}^n P_i$. Since $W$ consists of only nonzerodivisors on $M$ (as all associated primes are minimal over the annihilator of $M$), the map $M\to \oplus_{i=1}^n M_{P_i}$ is injective. 
  
  Each $M_{P_i}$ is a module of finite length over the local ring $T_{P_i}$. Therefore, $M_{P_i}$ embeds into a finite direct sum of copies of the injective hull $E_i$ of the residue field of $T_{P_i}$. We can track the image of the original module $M$ into each copy of the injective hull separately. So, for a fixed $i$ and a fixed copy of $E_i$, consider the composite map $M \to M_{P_i} \hookrightarrow E_i$.
  
  Since $T$ is Gorenstein, $T_{P_i}$ is a Gorenstein local ring of dimension $h$. Let $\seq{x}{h} \in P_i$ be elements in $T$ that form a system of parameters for $T_{P_i}$ (hence, they form a regular sequence in $T_{P_i}$). The injective hull $E_i$ over $T_{P_i}$ can be expressed as the direct limit of the $T_{P_i}$-modules $T_{P_i}/(\seq{x^s}{h})T_{P_i}$ as $s \to \infty$. 
  
  Crucially, $M$ is a finitely generated $T$-module. Therefore, its image in the direct limit $E_i$ must land inside a single term of the direct system for some large integer $s$. Thus, the map $M \to E_i$ factors through a $T$-module homomorphism $M \to T_{P_i}/(\seq{x^s}{h})T_{P_i}$. 
  
  Because $M \to E_i$ was injective when restricted to the component $M \to M_{P_i}$, the elements of $M$ are not killed by any element outside of $P_i$. Therefore, the map actually factors through the unlocalized quotient $T/(\seq{x^s}{h})T$. Since the sum of these maps over all $i$ and all copies of $E_i$ separates the points of $M$, we obtain an embedding of $M$ into a finite direct sum of modules of the form $T/(\seq{x^s}{h})T$. Since $\seq{x}{h}$ is a regular sequence in the Gorenstein local ring $T$, so is $\seq{x^s}{h}$, completing the proof.
\end{proof}

\begin{enumerate}[resume,left=0pt,label=(A\arabic*)]
\item Let $\mfa\subseteq \widehat{A}$ be an ideal of pure height $h$, then for
  $\mu\gg_{\loc} \nu$, the descent ideal $\mfa_{\mu}$ also has pure height $h$.\label{AR11}
  \begin{proof}
    This is equivalent to $\widehat{A}/\mfa$ having pure height $h$. By \cref{teza-AR-PureHeight}, this is equivalent to $\widehat{A}/\mfa$ injecting into a finite direct sum of modules obtained by killing a regular sequence in $\widehat{A}$. By \cref{teza-AR-RegSeqAlmostAll}, \ref{AR8} and \ref{AR10} we can descend the presentation of $\widehat{A}/\mfa$ as well as the injection map. Then $\mfa_{\mu}$ will have pure height $h$ as well for all $\mu\gg_{\loc}\nu$.
  \end{proof}
\end{enumerate}

\subsection{The complete local case}
We would like to prove the following result which is analogous to \cite[Corollary 2.4.10]{Huneke1999}.
\begin{thm}\label{teza-thm-TEforAnalytic}
  Suppose that $R$ is a reduced, equidimensional, complete local ring of dimension $d$ over $K$. Then the Jacobian ideal $\cJ(R/K)$ is contained in the test ideal for $K$-tight closure, and, hence, the test ideal for small equational tight closure.
\end{thm}
\begin{proof}
  By Cohen's structure theorem, the complete local ring $R$ has a coefficient field $L$ such that $K \subseteq L \cong R/\mathfrak{m}$. Because $K$ is a field of characteristic $0$, the field extension $L/K$ is separable. Consequently, the module of differentials $\Omega_{L/K} = 0$. From the fundamental exact sequence of modules of differentials 
  $$ \Omega_{L/K} \otimes_L R \to \Omega_{R/K} \to \Omega_{R/L} \to 0, $$
  we obtain an isomorphism $\Omega_{R/K} \cong \Omega_{R/L}$. Since the Jacobian ideal is the first nonzero Fitting ideal of the module of differentials, this implies $\cJ(R/K) = \cJ(R/L)$. Therefore, we may replace $K$ with $L$ and assume without loss of generality that $K$ is the coefficient field of $R$.
  
Suppose that we have a presentation $R=K \ps{\seq{x}{n}}/(\seq{f}{r})$. Then the Jacobian ideal is generated by $(n-d)\times (n-d)$ minors of the Jacobian matrix $\left( \partial f_i / \partial x_j \right)$. 

  Write $A=K[\seq{x}{n}]$ and then $\widehat{A}=K\ps{\seq{x}{n}}\twoheadrightarrow R$. Since $R$ is equidimensional, the kernel $(\seq{f}{r})$ has pure height $h=n-d$. Let $\delta$ be an $h\times h$ minor of $\left( \frac{\partial f_i}{\partial x_j} \right)$. Let $u\in I^*$ where $I\subseteq R$ an ideal.
  
  For each positive integer $N$, we aim to prove that $\delta u\in I+\mm^N$. Fix an $N$. Since each $f_i$ is a power series in $x_i$, we can truncate $f$ at degree $N$, i.e., let $f_i^{\leqslant N}$ be the sum of terms in $f_i$ of degree at most $N$ and each term in $f_i-f_i^{\leqslant N}$ is divisible by a $N+1$ power of $x_i$. So we can write
  \[
    f_i = f_i^{\leqslant N} + \sum_{\underline{\alpha}}\underline{x}^{\underline{\alpha}}u_{i,\underline{\alpha}}
  \]
  where $\underline{\alpha}\in \NN^n$ with $\abs{\underline{\alpha}}=N+1$, for some $u_{i,\underline{\alpha}}\in \widehat{A}$. Note that each $f_i^{\leqslant N}$ is in $A$.

  By \ref{AR2}, we can fix an index $\nu_0$ such that for all $\nu\geqslant \nu_0$, $A_{\nu}$ contains the generators of $I$, $u$ and all these $u_{i,\alpha}$. For each $\nu$, consider a presentation $A[\seq{y}{s}]\twoheadrightarrow A_{\nu}$. It has a kernel generated by $\seq{G}{t}$, i.e., $A_{\nu}\cong A[\seq{y}{s}]/(\seq{G}{t})$. Since all $u_{i,\underline{\alpha}}$ and $x_i$ are in $A_{\nu}$, we can write $F_i =  f_i^{\leqslant N} + \sum_{\underline{\alpha}}\underline{x}^{\underline{\alpha}}u_{i,\underline{\alpha}}$. Let $R_{\nu}=A_{\nu}/(\seq{F}{r})A_{\nu}=A[\seq{y}{s}]/(\seq{F}{r},\seq{G}{t})$. Since $(\seq{F}{r})$ is a descent of the ideal $(\seq{f}{r})$, by \ref{AR11}, the ideal $(\seq{F}{r})A_{\mu}$ also has pure height $h$ for $\mu\gg_{\loc}\nu$. Equivalently, $R_{\mu}$ is equidimensional for all $\mu\gg_{\loc}\nu$. Now consider the Jacobian matrix of $R_{\mu}$ over $k$. Let $h+s$ be the height of the ideal $(\seq{F}{r},\seq{G}{t})$ in the ring $A_{\mu}$. Then the Jacobian ideal $\cJ(R_{\mu}/K)$ is generated by the $h+s$ minors of the $(r+t)\times (n+s)$ matrix
  \begin{equation}\label{teza-eq-JacobianMatrix}
    \begin{pmatrix}
      \left( \frac{\partial F_i}{\partial x_j} \right)_{r\times n} & \left( \frac{\partial F_i}{\partial y_k}  \right)_{r\times s} \\
      \left( \frac{\partial G_l}{\partial x_j} \right)_{t\times n} & \left( \frac{\partial G_l}{\partial y_k}  \right)_{t\times s} 
    \end{pmatrix}.
  \end{equation}
  Since $u\in (IR)^{*K}$, there is some affine progenitor $R'$ such that this holds. We can make $A_{\mu}$ large enough to contain all the generators of $R'$ over $K$ to get a map $A'\to A_{\mu}$. Then we also have $u\in (IR_{\mu})^*$. By \cref{HH99}, the image in $R_{\mu}$ of the elements in $\cJ(R_{\mu}/k)$ multiplies the tight closure back into the ideal itself. We know that $\cJ(R_{\mu}/k)u\subseteq IR_{\mu}$.

  Note that in the matrix \eqref{teza-eq-JacobianMatrix}, the lower-right corner $\left( \frac{\partial G_l}{\partial y_k} \right)_{t\times s}$ is the Jacobian matrix of $A_{\mu}/A$. Since $A_{\mu}$ is smooth over $A$, the Jacobian ideal is the unit ideal.

  For each $F_i$, we have
  \begin{align*}
    \frac{\partial F_i}{\partial x_j}&=\frac{\partial f_i^{\leqslant N}}{\partial x_j} + \mm^N, \\
    \frac{\partial F_i}{\partial y_j}&=0+\sum_{\underline{\alpha}}\underline{x}^{\underline{\alpha}}\frac{\partial u_{i,\underline{\alpha}}}{\partial y_k}.
  \end{align*}
  So there is some $h\times h$ minor $\widetilde{\delta}$ of $\left( \frac{\partial F_i}{\partial x_j} \right)_{r\times n}$ such that $\widetilde{\delta}-\delta\in\mm^N$. Thinking of the matrix \eqref{teza-eq-JacobianMatrix} in the ring $R_{\mu}/\mm_{\mu}^N$ where $\mm_{\mu}$ is the image in $R_{\mu}$ of the descent of $\mm$ to $A_{\mu}$ by \ref{AR5}, we have
  \[
    \begin{pmatrix}
      \overline{J(R/K)} & 0\\
      * & \cQ
    \end{pmatrix},
  \]
where $\cQ$ is the image of the Jacobian matrix $J(A_{\mu}/A)$. Hence, the $s\times s$ minors of $\cQ$ generate the unit ideal. Since the product of any $h\times h$ minor of $\overline{J(R/K)}$ and $s\times s$ minor of $\cQ$ is in $\cJ(R_{\mu}/K)$, we have 
  \begin{align*}
  \overline{\delta}\cdot \cJ(R_{\mu}/K)&\subseteq \cJ(R_{\mu}/K)R_{\mu}/\mm_{\mu}^N,\\
  \widetilde{\delta}&\in \cJ(R_{\mu}/K)R_{\mu}/\mm_{\mu}^N.
  \end{align*}
  Therefore, we have $\widetilde{\delta} u \in (I +\mm_{\mu}^N)R_{\mu}$, which implies that $\delta u\in I+\mm^N$ in $R$. Since this is true for any $N$, we conclude that $\delta u\in \bigcap_N (I+\mm^N)R=I$.
\end{proof}


\subsection{The affine-analytic case}
We want to prove
\begin{thm}\label{tezs-thm-TEforSemianalytic}
  Suppose that $R$ is a reduced affine-analytic $K$-algebra that is equiheight. Then the Jacobian ideal $\cJ(R/K)$ is contained in the test ideal for $K$-tight closure, and, hence, the test ideal for small equational tight closure.
\end{thm}

We need to establish results (B1) - (B3) similar to (A1) - (A3). Let us begin with some discussion on the setup. Let $R=T/I$ where $T=K \ps{\seq{x}{n}}[\seq{z}{m}]$ and $I=(\seq{f}{r})T$. Then by assumption, all minimal primes of $I$ in $T$ are of the same height $h$. Write $\underline{z}$ for the sequence of elements $\seq{z}{m}$ and using the notation from previous section, we write $A=K[\seq{x}{n}]$ and $\widehat{A}=K\ps{\seq{x}{n}}$. Then $R=T/I$ where $T=\widehat{A}[\underline{z}]$. By hypothesis, all minimal primes of $I$ in $T$ have the same height and $I$ has no embedded primes. Suppose that $\hght(I)=h$ and $I=(\seq{f}{r})T$. Then the complete Jacobian matrix is given by
\[
\left(
\begin{array}{ccc:ccc}
\fpd[x_1]{f_1} & \cdots & \fpd[x_n]{f_1} & \fpd[z_1]{f_1} & \cdots & \fpd[z_m]{f_1}  \\
\vdots & \ddots & \vdots & \vdots & \ddots & \vdots \\
\fpd[x_1]{f_r} & \cdots & \fpd[x_n]{f_r} & \fpd[z_1]{f_r} & \cdots & \fpd[z_m]{f_r} 
\end{array}
\right).
\]

We will prove this using the Artin-Rotthaus theorem to approximate $\widehat{A}$ and also $T$. Therefore we will establish (A1)-(A11) as before for $T$. We first establish (A1)-(A5), summarized in the following bullet point.
\begin{enumerate}[left=0pt,label=(B\arabic*)]
    \item Note that any element in $\widehat{A}$ will be in the image of some $A_{\nu}$ and the map is eventually injective in the sense of \ref{AR1}, i.e., there exists some $\mu\geqslant \nu$ such that for all $\gamma \geqslant \mu$, the image of $A_{\nu}$ in $A_{\gamma}$ maps injectively into $\widehat{A}$. So any polynomial in $T$ will necessarily be in $T_{\nu}$. In particular, we can define $\seq{f}{r}$ in $T_{\nu}:=A_{\nu}[\underline{z}]$ and form $R_{\nu}=T_{\nu}/(\seq{f}{r})$.\label{B1}
    \item The statements in \ref{AR4}, \ref{AR5} and \ref{AR6} also hold here: we can descend an ideal by descending its generators. In particular, we can descend the maximal ideal. We can also descend an ideal and its radical. We will write $\mm$ for the maximal ideal of $\widehat{A}$ and $\mm_{\nu}$ for its contraction back to $A_{\nu}$. Then one is allowed to localize at any elements in $A_{\nu}-\mm_{\nu}$ for any $T_{\nu}$.\label{B2}
\end{enumerate}

Let us work with a counter-example to \cref{tezs-thm-TEforSemianalytic}, i.e., there is some ideal $J\subseteq R$ and $u\in R$ such that $u\in J^{*}$, and some $\delta\in\cJ(R/K)$ such that $\delta u\nin J$. Let $Q$ be a minimal prime of the proper ideal $J:_R\delta u$. This continues to be a counterexample in $R_Q$. We will make repeated use of the fact that we can localize at finitely many elements (in fact, one by localizing at the product of them) outside $Q$.

\begin{enumerate}[resume,left=0pt,label=(B\arabic*)]
    \item We will localize at finitely many elements in $T_{\nu}-Q^{\mathrm{c}}T_{\nu}$ where $Q^\mathrm{c}$ is the contraction of $Q$ in $T_{\nu}$. We will use the notation $\mu\gg_{Q-\loc}\nu$ to indicate this.\label{B3}
\end{enumerate}

We need the following lemma to deal with preserving height while descending.
\begin{lem}\label{tezs-B-Height}
  Let $Q\subseteq T$ be fixed and let $\seq{I}{k}$ be finitely many ideals contained in $Q$. Suppose that $\seq{J}{k}$ are ideals in $T_{\nu}$ such that $J_iT=I_i$, $1\leqslant i\leqslant k$. Then for all $\mu\gg_{Q-\loc}\nu$, all associated primes of each $J_iT_{\mu}$ are contained in the contraction of $Q$, and the height of $J_iT_{\mu}$ is the same as the height of $I_iT_Q$. If $\seq{g}{h}$ form a regular sequence in $QT_Q$, then for all $\mu\gg_{Q-\loc} \nu$, their images in $T_{\mu}$ also form a regular sequence.
\end{lem}
\begin{proof}
There are only finitely many ideals $\seq{J}{k}$ and each has finitely many associated primes. For each $J_i$, we can choose an element in all associated primes not contained in $Q^{\mathrm{c}}$ and avoiding associated primes contained in $Q^{\mathrm{c}}$. Then by localizing at these elements (or their product), we assume that all associated primes are contained in $Q^{\mathrm{c}}$.

  Let $\seq{g}{h}$ be a regular sequence in $QT_Q$. This implies that all associated primes of $(\seq{g}{i})$ contained in $Q$ have height $i$. For each associated prime $P$ of $(\seq{g}{i})T_{\nu}$ in $T_{\nu}$, since the height of $PT_Q$ is $i$ by \cref{pre-thm-RegHeight}, $P$ has height at least $i$. On the other hand, $P$ cannot have height more than $i$ due to the Krull's height theorem. So $P$ has height $i$. Since $T_{\mu}$ is Cohen-Macaulay, $\seq{g}{h}$ is a regular sequence.

  For preservation of height, we start with prime ideals. Let $P$ be a prime ideal in $T_{\nu}$. The height of $P$ cannot increase when expand to $T_Q$. Choose a maximal regular sequence in $PT_Q$. For all $\mu\gg_{Q-\loc}\nu$, they will form a regular sequence in $T_{\mu}$. Hence $\hght(PT_{\mu})=\hght(PT_Q)$.
  
  For general ideal $J$, we first choose $T_{\mu}$ where $\mu\gg_{Q-\loc}\nu$ such that all minimal primes of $JT_{\mu}$ are contained in $Q^{\mathrm{c}}$. We will also replace $J$ by its radical in $T_{\mu}$, i.e., the intersection of all minimal primes of $J$. Let $\seq{P}{l}$ be all the minimal primes of $J$. Then
  \[
  P_1T_Q\cap\cdots\cap P_lT_Q \subseteq JT_Q\subseteq P_iT_Q
  \]
  for each $i$. So we have $\hght(JT_Q)\leqslant \min\set{\hght(P_iT_Q}$. For any minimal prime of $JT_Q$, it must also contain some $P_iT_Q$. Therefore $\hght(JT_Q)\geqslant \min\set{\hght(P_iT_Q}$. Hence, $\hght(JT_Q)= \min\set{\hght(P_iT_Q}=\hght(J)$.
\end{proof}

\begin{proof}[Proof of \cref{tezs-thm-TEforSemianalytic}]
  Write $A=K[\seq{x}{n}]$ and $T=\widehat{A}[\seq{z}{m}]$. Then $T\twoheadrightarrow R$. The kernel $(\seq{f}{r})$ has pure height $h=n-d$. Let $\delta$ be a $h\times h$ minor of $\left( \frac{\partial f_i}{\partial x_j} \right)$. Let $u\in J^*$ where $J\subseteq R$ is an ideal.
  
  We work with a counterexample as before. Let $Q$ be a prime ideal containing $J:_R \delta u$. Then $\delta u\nin JR_Q$. Since $\mm\subseteq Q$, we have $\delta u\nin (J+\mm^N)R_Q$ for some $N$. Fix this $N$.
  
  We aim to prove that $\delta u\in I+\mm^N$. Since each $f_i$ is a polynomial in $z$ with coefficients being power series in $x_i$, we can truncate each coefficient of $f$ at degree $N$, i.e., if
  \[
  f_i(\seq{z}{m}) = \sum_{\text{finitely many }\underline{\beta}\in \NN^m} c_{\underline{\beta}} \underline{z}^{\underline{\beta}},
  \]
  where each $c_{\underline{\beta}}$ is in $\widehat{A}$, then we can form $c_{\underline{\beta}}^{\leqslant N}$ and the difference $c_{\underline{\beta}}-c_{\underline{\beta}}^{\leqslant N}$ is in $\mm^{N+1}T$. So we can write
  \begin{equation}\label{tezs-eq-ReductionForm}
      f_i(\seq{z}{m}) = \underbrace{\sum_{\text{finitely many }\underline{\beta}\in \NN^m} c_{\underline{\beta}}^{\leqslant N} \underline{z}^{\underline{\beta}}}_{f_i^{\leqslant N}} \quad + \quad r_i,
  \end{equation}
  for some $r_i\in \mm^{N+1}T$. Let $f_i^{\leqslant N}$ be the summation in \eqref{tezs-eq-ReductionForm}. The difference $r$ is the polynomial in $\seq{z}{m}$ with coefficients in $\mm$. So we can write
  \[
    f_i = f_i^{\leqslant N} + \sum_{\underline{\alpha}}\underline{x}^{\underline{\alpha}}u_{i,\underline{\alpha}}(\seq{z}{m}),
  \]
  where $\underline{\alpha}\in \NN^n$ with $\abs{\underline{\alpha}}=N+1$, for some $u_{i,\underline{\alpha}}(\seq{z}{m})\in T$. Note that each $f_i^{\leqslant N}$ is in $A[\seq{z}{m}]$.

  By \ref{B1}, we can fix an index $\nu_0$ such that for all $\nu\gg_{Q-\loc} \nu_0$, $T_{\nu}$ contains the generators of $J$, $u$ and all these $u_{i,\alpha}$. For each $\nu$, consider a presentation $A[\seq{y}{s}]\twoheadrightarrow A_{\nu}$. It has a kernel generated by $\seq{G}{t}$, i.e., $A_{\nu}\cong A[\seq{y}{s}]/(\seq{G}{t})$, and we have $T_{\nu}=A_{\nu}[\seq{z}{m}]=A[\seq{y}{s},\seq{z}{m}]/(\seq{G}{t})$.

  Since all $u_{i,\underline{\alpha}}(\seq{z}{m})$ and $x_i$ are in $T_{\nu}$, we can form $F_i =  f_i^{\leqslant N} + \sum_{\underline{\alpha}}\underline{x}^{\underline{\alpha}}u_{i,\underline{\alpha}}$. Let $R_{\nu}=T_{\nu}/(\seq{F}{r})T_{\nu}$. Since $(\seq{F}{r})$ is a descent of the ideal $(\seq{f}{r})$, by \cref{tezs-B-Height}, the ideal $(\seq{F}{r})A_{\mu}$ also has pure height $h$ for $\mu\gg_{\loc}\nu$. Now consider the Jacobian matrix of $R_{\mu}$ over $k$. Let $h+s$ be the height of the ideal $(\seq{F}{r},\seq{G}{t})$ in the ring $A_{\mu}$. Then the Jacobian ideals $\cJ(R_{\mu}/K)$ are generated by the $h+s$ minors of the $(r+t)\times (n+m+s)$ matrix
  \begin{equation}\label{tezs-eq-JacobianMatrix}
    \begin{pmatrix}
      \left( \frac{\partial F_i}{\partial x_j} \right)_{r\times n} & \left( \frac{\partial F_i}{\partial z_l} \right)_{r\times m} & \left( \frac{\partial F_i}{\partial y_k}  \right)_{r\times s} \\
      \left( \frac{\partial G_l}{\partial x_j} \right)_{t\times n} & \left( \frac{\partial G_l}{\partial z_l} \right)_{t\times m} & \left( \frac{\partial G_l}{\partial y_k}  \right)_{t\times s} 
    \end{pmatrix}.
  \end{equation}
  Since $u\in (JR)^{*K}$, there is some affine progenitor $R'$ such that $u\in (JR')^{*K}$ holds. We can make $T_{\mu}$ large enough to contain all the generators of $R'$ over $k$ to get a map $R'\to T_{\mu}$. Then we also have $u\in (JR_{\mu})^*$. By \cref{HH99}, the image in $R_{\mu}$ of the elements in $\cJ(R_{\mu}/K)$ multiplies the tight closure of any ideal back into the ideal. We know that $\cJ(R_{\mu}/K)u\subseteq JR_{\mu}$.

  Note that in the matrix \eqref{tezs-eq-JacobianMatrix}, the lower-right corner $\left( \frac{\partial G_l}{\partial y_k} \right)_{t\times s}$ is the Jacobian matrix of $A_{\mu}/A$. Since $A_{\mu}$ is smooth over $A$, the Jacobian ideal is a unit ideal.

  For each $F_i$, we have
  \begin{align*}
    \frac{\partial F_i}{\partial x_j}&=\frac{\partial f_i^{\leqslant N}}{\partial x_j} + \mm^N, \\
    \frac{\partial F_i}{\partial z_l}&=\frac{\partial f_i^{\leqslant N}}{\partial z_l} + \mm^{N+1}, \\
    \frac{\partial F_i}{\partial y_k}&=0+\sum_{\underline{\alpha}}\underline{x}^{\underline{\alpha}}\frac{\partial u_{i,\underline{\alpha}}}{\partial y_k}.
  \end{align*}

So there is some $h\times h$ minor $\widetilde{\delta}$ of $\left( \frac{\partial F_i}{\partial x_j} \right)_{r\times n}$ such that $\widetilde{\delta}-\delta\in\mm^N$. Thinking of the matrix \eqref{tezs-eq-JacobianMatrix} in the ring $R_{\mu}/\mm_{\mu}^N$ where $\mm_{\mu}$ is the image in $R_{\mu}$ of the descent of $\mm$ to $A_{\mu}$ by \ref{B2}, we have
  \[
    \begin{pmatrix}
      \overline{J(R/K)} & 0\\
      * & \cQ
    \end{pmatrix},
  \]
  where $\cQ$ is the image of the Jacobian matrix $J(A_{\mu}/A)$. Since the product of any $s\times s$ minor of $\cQ$ and any $h\times h$ minor of $\overline{J(R/K)}$ is in the image of $\cJ(R_{\mu}/K)$, lifting this relation back to $R_{\mu}$ gives
  \[
    \widetilde{\delta} \cJ(A_{\mu}/A) R_{\mu} \subseteq \cJ(R_{\mu}/K) + \mm_{\mu}^N R_{\mu}.
  \]
  Because $A_{\mu}$ is smooth over $A$, the Jacobian ideal $\cJ(A_{\mu}/A)$ is the unit ideal in $A_{\mu}$. Thus, $\cJ(A_{\mu}/A) R_{\mu} = R_{\mu}$, which yields
  \[
    \widetilde{\delta} \in \cJ(R_{\mu}/K) + \mm_{\mu}^N R_{\mu}.
  \]
  Therefore, we have $\widetilde{\delta} u \in (I +\mm_{\mu}^N)R_{\mu}$, which implies that $\delta u\in I+\mm^N$ in $R$. So we obtain a contradiction! We conclude that $\delta u\in J$.
\end{proof}

\subsection{The semianalytic case}
We want to show:
\begin{thm}\label{tczs-thm-GeoRegFlatExt}
Suppose that we have a flat map with geometrically regular fibers $R\to R'$ where $R$ is a reduced affine-analytic equiheight $K$-algebra. Then the expansion of the Jacobian ideal $\cJ(R/K)R'$ is contained in the test ideal of $R'$ for $K$-tight closure, and, hence, the test ideal for small equational tight closure.
\end{thm}

Since $R$ is a reduced affine-analytic $K$-algebra, it is approximately Gorenstein. We need to show that $R'$ is also approximately Gorenstein so that \cref{dtc-conv-TestIdeal} makes sense. Hence, we prove the following proposition.

\begin{prop}\label{tczs-prop-AppGro}
Let $S\to T$ be flat with Gorenstein fibers. If every local ring of $S$ is approximately Gorenstein (this condition holds, for example, if $S$ is excellent and reduced), then $T$ is also approximately Gorenstein.
\end{prop}
\begin{proof}
Let $\mfn$ be a maximal ideal of $T$ and let $\mfq$ be its contraction in $S$. Then $S_{\mfq}\to T_{\mfn}$ is local and flat, with Gorenstein fibers. So we reset the notation and assume that $(S,\mfq)\to (T,\mfn)$ is flat local with Gorenstein fibers. By assumption $S$ is approximately Gorenstein. So there exists a sequence of irreducible ideals $I_t$ in $S$ cofinal with powers of $\mfq$. Let $\seq{x}{h}$ be a system of parameters in $T/\mfq T$. We claim that $I_tT+(\seq{x^t}{h})$ is a sequence of irreducible ideals in $T$ cofinal with powers of $\mfn$.

The ``cofinal'' part is trivial from the construction. For irreduciblity, note that $S/I_t$ is Gorenstein. Hence, $T/I_t T$ is Gorenstein as it is flat local over $S/I_t$ with Gorenstein fibers. By construction, $\seq{x^t}{h}$ is a regular sequence on $T/I_t T$. Hence, $T/(I_tT+(\seq{x^t}{h}))$ is Gorenstein, which implies that $I_tT+(\seq{x^t}{h})$ is irreducible.
\end{proof}

\begin{proof}[Proof of \cref{tczs-thm-GeoRegFlatExt}]
By \cref{tczs-prop-AppGro} we know that $R'$ is also approximately Gorenstein. By \cref{teza-thm-FlatGeoRegFiber}, $R'$ is the filtered direct limit of smooth $R$-algebras. Since any smooth extension $S$ of $R$ is an affine-analytic $K$-algebra, we only need to show that the $\cJ(R/K)S\subseteq \cJ(S/K)$. Then \cref{tezs-thm-TEforSemianalytic} will finish the proof.

Since an element is in an ideal if and only if it is so over each connected component, if $\Spec(R)$ has several connected components, we can deal with each component separately. So we assume that $\Spec(R)$ is connected. Let $S$ be a smooth extension of $R$. Again if $\Spec(S)$ has multiple connected components, we can deal with each component separately. So we assume that $\Spec(S)$ is connected as well.

We can find finitely many elements $g_j \in R$, $f_j$ in $S$ such the $f_jg_j$ generate the unit ideal of $S$, the $g_j$ generate the unit ideal of $R$, and each $S_{f_jg_j}$ is a special smooth extension of $R_{g_j}$., i.e., \'etale over a polynomial extension by \cite[Chapter III, Theorem 2.1]{Iversen1973}.

Let us deal with each piece $R_{g_i}\to S_{f_ig_i}$ separately. We write $R_i,S_i$ for $R_{g_i},S_{f_ig_i}$. Then $S_i$ is standard \'etale over a polynomial ring $T_i$ over $R_i$. We can write $S_i\cong (T_i[X]/H(X))_G$ where $H(X)$ is monic in $T_i$, and $G$ a multiple of $H'(X)$. If $H(X)$ in any minimal prime of $(T_i[X])_G$, say $\mfq$, then $\mfp:=\mfq\cap T_i$ is also a minimal prime and since $\mfp (T_i[X])_G$ is a prime. We conclude that $\mfq=\mfp (T_i[X])_G$. Then $H(X)$ has all its coefficients in $\mfp$, which implies that $H'(X)\in\mfp (T_i[X])_G=\mfq$. But this is a contradiction as $H'(X)$ is a unit in $(T_i[X])_G$.

If we work with a presentation $R=T/I$ of $R$ where $T=K\ps{\seq{x}{n}}[\seq{z}{m}]$, then we can assume that $I=(\seq{f}{r})$ has pure height $h$ in $T$ and $S=T[\seq{y}{\ell}]/I'$ where $I'=(\seq{f}{r},\seq{g}{s})$. Since $\Spec(S)$ is connected, the above argument shows that all minimal primes of $I'$ in $S$ will have the same height, which we denote by $h+t$. Then we can write the Jacobian matrix blockwise as
\[
\begin{pmatrix}
\fpd[x_j]{f_i} & \fpd[z_k]{f_i} & 0 \\
* & * & \fpd[y_a]{g_b}
\end{pmatrix}.
\]
The $s\times s$ minors of the right bottom block $(\fpd[y_a]{g_b})$ generate $\cJ(S/R)$, which is the unit ideal. The Jacobian ideal $\cJ(S/K)$ is generated by $h+s$ minors of this Jacobian matrix. Hence any $h\times h$ minor of the block $\left(\fpd[x_j]{f_i} \quad \fpd[z_k]{f_i}\right)$, multiplies $\cJ(S/K)$ into the Jacobian ideal $\cJ(S/K)$, which implies that $\cJ(R/K)S\subseteq \cJ(S/K)$.
\end{proof}

\begin{cor}\label{tczs-cor-Semianalytic}
If $R$ is a semianalytic $K$-algebra that is the localization of a reduced equiheight affine-analytic $K$-algebra, then the Jacobian ideal $\cJ(R/K)$ is contained in the test ideal. For any flat $K$-algebra morphism $R\to R'$ with geometrically regular fibers, the expansion of the Jacobian ideal $\cJ(R/K)R'$ is contained in the test ideal of $R'$. Here, all test ideals are for $K$-tight closure, which are contained in the corresponding test ideals for small equational tight closure.
\end{cor}
\begin{proof}
Write $R=W^{-1}T/I$ where $T=K\ps{\seq{x}{n}}[\seq{z}{m}]$. The assumption is that $I$ has pure height $h$ in $T$. Since $T/I\to R$ is a flat map with geometrically regular fibers, using \cref{tczs-thm-GeoRegFlatExt}, we conclude that the Jacobian ideal expanded $\cJ((T/I) / K)R$ is contained in the test ideal. Since $R$ is a localization of $T/I$, we have $\cJ(R/K)=\cJ((T/I) / K)R$.

The composition of the maps $T/I\to R\to R'$ is still flat. Since $\Spec(R)$ is a subset of $\Spec(T/I)$, the composition map $T/I\to R'$ has geometrically regular fibers. Hence, by \cref{tczs-thm-GeoRegFlatExt}, the Jacobian ideal expanded $\cJ((T/I)/K)R'$ is contained in the test ideal of $R'$. Since $\cJ((T/I)/K) R'=\cJ(R/K)R'$, the conclusion is proved.
\end{proof}

\subsection{The affine over excellent local case}
\begin{thm}\label{teza-thm-GeoFlatPower}
Let $K$ be a field of characteristic $0$, and let $A$ be a noetherian excellent local $K$-algebra. Let $R$ be an equiheight $A$-algebra of finite type, and let $S$ be geometrically regular over $R$. Suppose that $c \in S$ is an element such that $S_c$ is regular. Then $c$ has a power that is a test element for $K$-tight closure in $S$.
\end{thm}
\begin{proof}
It suffices to prove that $c$ is a test element after base change to the completion of $A$, which preserves the hypotheses (see, e.g., \cref{pre-thm-EquiheightPreserved}), so this reduces at once to the case where $A$ is complete. Because $K$ is a field of characteristic $0$, it is a perfect field. Since $R$ is essentially of finite type over $K$, absolute regularity coincides with regularity, and the Jacobian ideal $\cJ(R/K)$ defines the singular locus in $R$. Hence, its extension $\cJ(R/K)S$ defines the singular locus in $S$. Therefore, $c$ has a power in the extended Jacobian ideal $\cJ(R/K)S$, and so is a test element for $K$-tight closure.
\end{proof}

\begin{cor}\label{teza-cor-EssAffineOverExcellent}
Let $K$ be a field of characteristic $0$, and let $A$ be a noetherian excellent local $K$-algebra. Let $R$ be an equiheight $A$-algebra essentially of finite type. Suppose that $c \in R$ is an element such that $R_c$ is regular. Then $c$ has a power that is a test element for $K$-tight closure in $R$.
\end{cor}
\begin{proof}
Let $R'$ be an affine algebra over $A$ such that $R=W^{-1}R'$. Since $R$ is equiheight, we can localize at one more element to assume that $R'$ is also equiheight. Then the localization map $R'\to R$ is flat with geometrically regular fibers. The conclusion then follows directly from \cref{teza-thm-GeoFlatPower}.
\end{proof}

\bibliography{tc}

@article{Buchsbaum1973,
author = {Buchsbaum, David A. and Eisenbud, David},
doi = {10.1016/0021-8693(73)90044-6},
issn = {00218693},
journal = {Journal of Algebra},
month = {may},
number = {2},
pages = {259--268},
title = {{What makes a complex exact?}},
url = {https://linkinghub.elsevier.com/retrieve/pii/0021869373900446},
volume = {25},
year = {1973}
}

@book{Kunz1986,
address = {Wiesbaden},
author = {Kunz, Ernst},
doi = {10.1007/978-3-663-14074-0},
isbn = {978-3-528-08973-3},
publisher = {Vieweg+Teubner Verlag},
series = {Advanced Lectures in Mathematics},
title = {{K{\"{a}}hler Differentials}},
url = {http://link.springer.com/10.1007/978-3-663-14074-0},
year = {1986}
}

@article{Swan1998,
author = {Swan, Richard G.},
journal = {Algebra and geometry},
pages = {135--192},
publisher = {Internat. Press Cambridge, MA},
title = {{N{\'{e}}ron-Popescu desingularization}},
volume = {2},
year = {1998}
}

@book{Serre1975,
address = {Berlin, Heidelberg},
author = {Serre, Jean-Pierre},
booktitle = {Alg{\`{e}}bre Locale {\textperiodcentered} Multiplicit{\'{e}}s},
doi = {10.1007/978-3-540-37123-6},
isbn = {978-3-540-07028-3},
publisher = {Springer Berlin Heidelberg},
series = {Lecture Notes in Mathematics},
title = {{Alg{\`{e}}bre Locale {\textperiodcentered} Multiplicit{\'{e}}s}},
url = {http://link.springer.com/10.1007/978-3-540-37123-6},
volume = {11},
year = {1975}
}

@article{Avramov1994,
author = {Avramov, L.L. and Foxby, H.B. and Herzog, B.},
doi = {10.1006/jabr.1994.1057},
issn = {00218693},
journal = {Journal of Algebra},
month = {feb},
number = {1},
pages = {124--145},
title = {{Structure of local homomorphisms}},
url = {https://linkinghub.elsevier.com/retrieve/pii/S002186938471057X},
volume = {164},
year = {1994}
}

@article{Huneke1999,
author = {Hochster, Melvin and Huneke, Craig},
pages = {1--204},
title = {{Tight closure in equal characteristic zero}},
year = {1999},
journal={preprint}
}

@incollection{Artin1988,
author = {Artin, Michael and Rotthaus, Christel},
booktitle = {Algebraic Geometry and Commutative Algebra},
doi = {10.1016/B978-0-12-348031-6.50009-7},
pages = {35--44},
publisher = {Elsevier},
title = {{A structure theorem for power series rings}},
url = {https://linkinghub.elsevier.com/retrieve/pii/B9780123480316500097},
year = {1988}
}

@book{nagata1975local,
  title={Local Rings},
  author={Nagata, M.},
  isbn={9780882752280},
  lccn={74030409},
  series={Interscience tracts in pure and applied mathematics},
  year={1975},
  publisher={Krieger}
}

@article{Hochster1974,
author = {Hochster, Melvin and Roberts, Joel L.},
doi = {10.1016/0001-8708(74)90067-X},
issn = {00018708},
journal = {Advances in Mathematics},
month = {jun},
number = {2},
pages = {115--175},
title = {{Rings of invariants of reductive groups acting on regular rings are Cohen-Macaulay}},
url = {https://linkinghub.elsevier.com/retrieve/pii/000187087490067X},
volume = {13},
year = {1974}
}

@article{Hochster2002,
author = {Hochster, Melvin},
doi = {10.4310/HHA.2002.v4.n2.a14},
issn = {15320073},
journal = {Homology, Homotopy and Applications},
number = {2},
pages = {295--314},
title = {{Presentation depth and the Lipman-Sathaye Jacobian theorem}},
url = {http://www.intlpress.com/site/pub/pages/journals/items/hha/content/vols/0004/0002/a014/},
volume = {4},
year = {2002}
}

@article{Hochster2002b,
archivePrefix = {arXiv},
arxivId = {math/0211174},
author = {Hochster, Melvin and Huneke, Craig},
doi = {10.1007/s002220100176},
eprint = {0211174},
issn = {00209910},
journal = {Inventiones Mathematicae},
number = {2},
pages = {349--369},
primaryClass = {math},
title = {{Comparison of symbolic and ordinary powers of ideals}},
volume = {147},
year = {2002}
}

@article{Perez2019,
  title={Characteristic-free test ideals},
  author={P{\'e}rez, Felipe and RG, Rebecca},
  journal={Transactions of the American Mathematical Society, Series B},
  volume={8},
  number={24},
  pages={754--787},
  year={2021}
}

@article{Hochster1990a,
author = {Hochster, Melvin and Huneke, Craig},
doi = {10.2307/1990984},
issn = {08940347},
journal = {Journal of the American Mathematical Society},
month = {jan},
number = {1},
pages = {31},
title = {{Tight closure, invariant theory, and the Briancon-Skoda theorem}},
url = {http://www.jstor.org/stable/1990984?origin=crossref},
volume = {3},
year = {1990}
}

@incollection{Schwede2012,
address = {Berlin, Boston},
archivePrefix = {arXiv},
arxivId = {1104.2000},
author = {Schwede, Karl and Tucker, Kevin},
booktitle = {Progress in Commutative Algebra 2},
doi = {10.1515/9783110278606.39},
eprint = {1104.2000},
pages = {1--44},
publisher = {DE GRUYTER},
title = {{A survey of test ideals}},
url = {https://www.degruyter.com/view/books/9783110278606/9783110278606.39/9783110278606.39.xml},
year = {2012}
}

@article{Lyubeznik2001,
author = {Lyubeznik, Gennady and Smith, Karen E.},
doi = {10.1090/S0002-9947-01-02643-5},
journal = {Transactions of the American Mathematical Society},
month = {jan},
number = {08},
pages = {3149--3181},
title = {{On the commutation of the test ideal with localization and completion}},
url = {http://www.ams.org/journal-getitem?pii=S0002-9947-01-02643-5},
volume = {353},
year = {2001}
}

@article{Lyubeznik1999,
author = {Lyubeznik, Gennady and Smith, Karen E.},
doi = {10.1353/ajm.1999.0042},
issn = {1080-6377},
journal = {American Journal of Mathematics},
number = {6},
pages = {1279--1290},
title = {{Strong and weak F-regularity are equivalent for graded rings}},
url = {http://www.jstor.org/stable/25098970 http://muse.jhu.edu/content/crossref/journals/american{\_}journal{\_}of{\_}mathematics/v121/121.6lyubeznik.pdf},
volume = {121},
year = {1999}
}

@article{Hochster1977,
author = {Hochster, Melvin},
doi = {10.2307/1997914},
issn = {00029947},
journal = {Transactions of the American Mathematical Society},
month = {aug},
number = {2},
pages = {463},
title = {{Cyclic purity versus purity in excellent noetherian rings}},
url = {https://www.jstor.org/stable/1997914?origin=crossref},
volume = {231},
year = {1977}
}

@article{Mehta1985,
author = {Mehta, V. B. and Ramanathan, A.},
doi = {10.2307/1971368},
journal = {The Annals of Mathematics},
month = {jul},
number = {1},
pages = {27},
publisher = {JSTOR},
title = {{Frobenius splitting and cohomology vanishing for Schubert varieties}},
url = {https://www.jstor.org/stable/1971368?origin=crossref},
volume = {122},
year = {1985}
}

@article{Ramanan1985,
author = {Ramanan, S. and Ramanathan, A.},
doi = {10.1007/BF01388970},
issn = {0020-9910},
journal = {Inventiones Mathematicae},
month = {jun},
number = {2},
pages = {217--224},
publisher = {Springer-Verlag},
title = {{Projective normality of flag varieties and Schubert varieties}},
url = {https://link.springer.com/article/10.1007/BF01388970 http://link.springer.com/10.1007/BF01388970},
volume = {79},
year = {1985}
}

@article{Hochster1994,
author = {Hochster, Melvin and Huneke, Craig},
doi = {10.2307/2154942},
journal = {Transactions of the American Mathematical Society},
month = {nov},
number = {1},
pages = {1},
title = {{F-regularity, test elements, and smooth base change}},
url = {https://www.jstor.org/stable/2154942?origin=crossref},
volume = {346},
year = {1994}
}

@article{Hochster1989,
author = {Hochster, Melvin and Huneke, Craig},
doi = {10.24033/msmf.343},
issn = {0249-633X},
journal = {M{\'{e}}moires de la Soci{\'{e}}t{\'{e}} Math{\'{e}}matique de France},
pages = {119--133},
title = {{Tight closure and strong F-regularity}},
url = {http://www.numdam.org/item?id=MSMF{\_}1989{\_}2{\_}38{\_}{\_}119{\_}0},
volume = {1},
year = {1989}
}

@article{Hara2004,
archivePrefix = {arXiv},
arxivId = {arXiv:math/0210131v2},
author = {Hara, Nobuo and Takagi, Shunsuke},
eprint = {0210131v2},
journal = {Nagoya Mathematical Journal},
pages = {1--11},
primaryClass = {arXiv:math},
title = {{On a generalization of test ideals}},
volume = {175},
year = {2004}
}

@article{Hara2002,
  title={A generalization of tight closure and multiplier ideals},
  author={Hara, Nobuo and Yoshida, Ken-Ichi},
  journal={Transactions of the American Mathematical Society},
  volume={355},
  number={8},
  pages={3143--3174},
  year={2003}
}

@article{Schwede2018,
  title={Singularities in mixed characteristic via perfectoid big Cohen--Macaulay algebras},
  author={Ma, Linquan and Schwede, Karl},
  journal={Duke Mathematical Journal},
  volume={170},
  number={13},
  pages={2815--2890},
  year={2021},
  publisher={Duke University Press}
}

@article{Ma2018,
archivePrefix = {arXiv},
arxivId = {1705.02300},
author = {Ma, Linquan and Schwede, Karl},
doi = {10.1007/s00222-018-0813-1},
eprint = {1705.02300},
isbn = {1097-0215 (Electronic)$\backslash$r0020-7136 (Linking)},
issn = {0020-9910},
journal = {Inventiones mathematicae},
month = {nov},
number = {2},
pages = {913--955},
pmid = {23526433},
publisher = {Springer Berlin Heidelberg},
title = {{Perfectoid multiplier/test ideals in regular rings and bounds on symbolic powers}},
url = {https://doi.org/10.1007/s00222-018-0813-1 http://link.springer.com/10.1007/s00222-018-0813-1},
volume = {214},
year = {2018}
}

@article{Smith2000,
author = {Smith, Karen E.},
doi = {10.1080/00927870008827196},
issn = {0092-7872},
journal = {Communications in Algebra},
month = {jan},
number = {12},
pages = {5915--5929},
title = {{The multiplier ideal is a universal test ideal}},
url = {http://www.tandfonline.com/doi/abs/10.1080/00927870008827196},
volume = {28},
year = {2000}
}

@article{Hara2001,
author = {Hara, Nobuo},
doi = {10.1090/S0002-9947-01-02695-2},
issn = {0002-9947},
journal = {Transactions of the American Mathematical Society},
month = {jan},
number = {5},
pages = {1885--1906},
title = {{Geometric interpretation of tight closure and test ideals}},
url = {http://www.ams.org/tran/2001-353-05/S0002-9947-01-02695-2/},
volume = {353},
year = {2001}
}

@article{Epstein2019,
  title={Closure-interior duality over complete local rings},
  author={Epstein, Neil and RG, Rebecca},
  journal={Rocky Mountain Journal of Mathematics},
  volume={51},
  number={3},
  pages={823--853},
  year={2021},
  publisher={Rocky Mountain Mathematics Consortium Tempe, AZ, USA}
}

@article{Heitmann2001,
author = {Heitmann, Raymond C.},
doi = {10.1006/jabr.2000.8661},
issn = {00218693},
journal = {Journal of Algebra},
month = {apr},
number = {2},
pages = {801--826},
title = {{Extensions of plus closure}},
url = {http://linkinghub.elsevier.com/retrieve/pii/S0021869300986617},
volume = {238},
year = {2001}
}

@book{Iversen1973,
address = {Berlin, Heidelberg},
author = {Iversen, Birger},
booktitle = {Lecture Notes in Mathematics},
doi = {10.1007/BFb0060790},
isbn = {978-3-540-06137-3},
publisher = {Springer Berlin Heidelberg},
series = {Lecture Notes in Mathematics},
title = {{Generic Local Structure of the Morphisms in Commutative Algebra}},
url = {http://link.springer.com/10.1007/BFb0060790},
volume = {310},
year = {1973}
}
\bibliographystyle{halpha}

\end{document}